\documentclass[12pt]{article}

\usepackage{amssymb}
\usepackage{amsthm}
\usepackage{amsfonts}
\usepackage{amsmath}
\usepackage[style=alphabetic,maxnames=99,maxalphanames=5, isbn=false, giveninits=true, doi=false, url=false, eprint=true]{biblatex}
\renewbibmacro{in:}{}

\DeclareFieldFormat[article]{citetitle}{#1}
\DeclareFieldFormat[article]{title}{#1}
\DeclareFieldFormat[inbook]{citetitle}{#1}
\DeclareFieldFormat[inbook]{title}{#1}
\DeclareFieldFormat[incollection]{citetitle}{#1}
\DeclareFieldFormat[incollection]{title}{#1}
\DeclareFieldFormat[inproceedings]{citetitle}{#1}
\DeclareFieldFormat[inproceedings]{title}{#1}
\DeclareFieldFormat[phdthesis]{citetitle}{#1}
\DeclareFieldFormat[phdthesis]{title}{#1}
\DeclareFieldFormat[misc]{citetitle}{#1}
\DeclareFieldFormat[misc]{title}{#1}
\DeclareFieldFormat[book]{citetitle}{#1}
\DeclareFieldFormat[book]{title}{#1} 
\usepackage{mathtools}
\usepackage{hyperref}
\usepackage[capitalize]{cleveref}
\hypersetup{
colorlinks=false,
linktoc=all
}
\usepackage{enumitem}
\usepackage{hyphenat}
\usepackage[a4paper, margin=1 in]{geometry}
\usepackage[T1]{fontenc}
\usepackage[final]{microtype}
\usepackage{tikz-cd}

\newtheorem*{theorem*}{Theorem}

\newtheorem{theorem}{Theorem}[section]

\newtheorem{lemma}[theorem]{Lemma}
\newtheorem{corollary}[theorem]{Corollary}

\newtheorem{remark}[theorem]{Remark}
\newtheorem{example}[theorem]{Example}
\newtheorem{definition}[theorem]{Definition}
\newtheorem{proposition}[theorem]{Proposition}

\def\sim{\simeq}
\def\epsilon{\varepsilon}
\def\subset{\subseteq}
\newcommand{\PSH}{\mathrm{PSH}}
\newcommand{\ddc}{\mathrm{dd}^{\mathrm{c}}}

\DeclareMathOperator{\vol}{vol}
\DeclareMathOperator{\Tr}{Tr}
\newcommand{\QPSH}{\mathrm{QPSH}}
\DeclareMathOperator{\Rest}{Tr}

\title{The trace operator of quasi-plurisubharmonic functions on compact K\"ahler manifolds}
\author{Tam\'as Darvas, Mingchen Xia}
\date{}

\bibliography{ComplexGeometry}

\begin{document}

\maketitle
\begin{abstract}
We introduce the trace operator for quasi-plurisubharmonic functions on compact K\"ahler manifolds, allowing to study the singularities of such functions along submanifolds where their generic Lelong numbers vanish. Using this construction we obtain novel $L^2$ extension theorems and give applications to  restricted volumes of big line bundles.
\end{abstract}

\section{Introduction}

Let $(X,\omega)$ be a connected compact K\"ahler manifold. By $\QPSH(X)$ we denote the set of quasi-plurisubharmonic functions on $X$. These functions are locally the sum of a plurisubharmonic function and a smooth function.

We fix a smooth real closed $(1,1)$-form $\theta$ on $X$,  representing a big cohomology class $\{\theta\} \in H^{1,1}(X,\mathbb R)$. Let $\PSH(X,\theta) \subset \QPSH(X)$ be the space of $\theta$-psh functions, satisfying the condition $\theta_u=\theta + \ddc u \geq 0$, where $\ddc u \coloneqq \frac{\mathrm{i}}{2\pi} \partial \overline{\partial} u$. 

Suppose that $Y \subset X$ is a connected complex submanifold and $u \in \PSH(X,\theta)$ has its generic Lelong number vanishing along $Y$. From the point of view of algebraic geometry this means that $u$ is ``as good as bounded'' at a generic point of $Y$. However, there are plenty of examples where such $u$ is identically equal to $-\infty$ on $Y$. Due to this apparent anomaly, great enigma surrounds the singularity of such potentials $u$ near $Y$. In this work we make steps towards a better understanding, by attaching a `trace' potential $\Rest_Y^\theta(u) \in \PSH(Y,\theta|_Y)$ to such potential $u$ (see \eqref{eq: Tr_general_def_intr} for the precise definition).

Evidencing the usefulness of this concept we  provide a concise application at the interface of algebraic and differential geometry, after obtaining novel $L^2$ extension theorems.
More specifically, when $\theta$ is the curvature form of a big Hermitian line bundle $(L,h)$, we prove that the non-pluripolar volume of $\Rest_Y^\theta(u)$ is the same as the algebraic restricted volume of $u$ to $Y$ (see \eqref{rest_vol_intr}).

\paragraph{The trace of \texorpdfstring{$\theta$}{theta}-psh potentials.}

Given $u \in \PSH(X,\theta)$ recall the following envelope defined in \cite{RWN14}, inspired by ideas of Rashkovskii--Sigurdsson in the local case \cite{RS05}:
\[
P_\theta[u] \coloneqq  \sup\left\{v \in \PSH(X,\theta) :  v\leq 0, v \leq u +C \textup{ for some constant } C \right\} \in \PSH(X,\theta).
\]

Suppose that $Y \subseteq X$ is a connected $m$-dimensional complex submanifold, and $u \in  \PSH(X,\theta)$ such that $\theta_u$ is a K\"ahler current and the generic Lelong number $\nu(u,Y)=0$. Recall that $\nu(u,Y)$ is  the infimum of the Lelong numbers of $u$ along points of $Y$.

Let $u_j \in \PSH(X,\theta)$ be a decreasing quasi-equisingular approximation of $u$, whose existence was proved by Demailly (see Theorem~\ref{thm:Demailly}). Due to $\nu(u,Y)=0$, we get that $u_j
|_Y \not\equiv -\infty$, allowing us to define the trace of $u$ along $Y$:
\[
\Rest^\theta_Y(u) \coloneqq \lim_{j\to\infty} P_{\theta|_Y} [u_j|_Y] \in \PSH(Y,\theta|_Y).
\]
The above decreasing limit exists, because $\sup_X P_{\theta|_Y + \varepsilon_j \omega|_Y} [u_j|_Y]=0$. More importantly, we will show that  $\Rest^\theta_Y(u)$ does not depend on the choice of approximating sequence $u_j$ (\cref{lem: indep_quasi_Tr}).
In case $u \in  \PSH(X,\theta)$, but $\theta_u$ is not a K\"ahler current, we will extend the above definition in the following manner: 
\begin{equation}\label{eq: Tr_general_def_intr}
\Rest^\theta_Y(u) \coloneqq \lim_{\varepsilon \searrow 0} \Rest^{\theta + \varepsilon \omega}_Y(u) \in \PSH(Y,\theta|_Y).
\end{equation}
This definition does not depend on the choice of $\omega$. 
The singularity type of $\Rest^\theta_Y(u)$, unlike the potential $\Rest^\theta_Y(u)$ itself, depends only on the current $\theta_u$, not on how we decompose the current into $\theta$ and $\ddc u$. 
As a guiding principle, all meaningful quantities involving the trace operator should only depend on its singularity type. 

The terminology is inspired from the loose analogy with the classical trace operator in the theory of Sobolev spaces. To be more precise, recall that if $\Omega$ is a bounded domain in $\mathbb{R}^n$ with smooth boundary, the restriction $C^{\infty}(\bar{\Omega})\rightarrow C^{\infty}(\partial \Omega)$ admits a natural continuous extension $\Tr:W^{1,p}(\Omega)\rightarrow W^{1-1/p,p}(\partial \Omega)$ for any $p>1$.
In our setting, $X$ and $Y$ are the analogues of $\Omega$ and $\partial \Omega$.
One could regard potentials $u$ with analytic singularity type as the analogues of $C^{\infty}(\bar{\Omega})$, as in this case the singularity types of $\Rest^\theta_Y(u)$ and $u|_Y$ are the same (see Example~\ref{ex:analytic}). The general potentials $u \in \PSH(X,\theta)$ are analogues of Sobolev functions in this picture.

\paragraph{Applications to restricted volumes.} 
The (algebraic) restricted volume goes back to work of Tsuji \cite{Ts06} with thorough initial treatments by \cite{ELMNP, Mat13}, and spectacular applications in \cite{HM06,Tak06,BFJ09} etc.  

We start with a more general approach to this concept. Given $u \in \PSH(X,\theta)$, by $\mathcal I(u) \subset \mathcal O_X$ we denote the multiplier ideal sheaf of $u$, namely, the coherent ideal sheaf of holomorphic functions $f$ whose germs satisfy $\int |f|^2 \mathrm{e}^{-u}\omega^n < \infty$.

Let $(L,h)$ be a big line bundle on $X$ with curvature form $\theta \coloneqq \frac{\mathrm{i}}{2\pi} \Theta(h)$. The restricted volume of $u$ to $Y$, with respect to $L$ is given by the following quantity:
\begin{equation}\label{rest_vol_intr}
\vol_{X|Y}(L,u) \coloneqq \varlimsup_{k\to\infty} \frac{m!}{k^m} \dim_{\mathbb C}\left\{ s|_Y:  s\in H^0(X,\mathcal{O}_X(L)^k \otimes \mathcal I(k u))\right\}.
\end{equation}

Clearly, $\vol_{X|Y}(L,u) =0$ if $\nu(u,Y) >0$. So this concept is meaningful only in case $\nu(u,Y) =0$, when $\Rest_Y^\theta(u)$ is defined. As we will see shortly, this is not a coincidence.

One can also define a notion of restricted volume for $u$ using $H^0(Y,\mathcal{O}_Y(L)^k \otimes \mathcal I(ku)|_Y)$, the space of global sections of $L|_Y^k$ that have local extensions that are square integrable with respect to $\mathrm{e}^{-ku}$. However, thanks to the next theorem, this notion turns out to recover the same restricted volume defined in  \eqref{rest_vol_intr}. More importantly, all these notions of restricted volume agree with 
\[
\int_Y\left(\theta|_Y+\ddc \Rest_Y^\theta(u)\right)^m,
\]
the non-pluripolar volume of $\Rest_Y^\theta(u)$ defined in the sense of \cite{BEGZ10}. In addition to all this, we also obtain that the limit exists in the definition of $\vol_{X|Y}(L,u)$ in \eqref{rest_vol_intr}: 

\begin{theorem}[{\cref{thm: relativeDX} and \cref{thm: rest_volume_2}}]\label{thm: rest_volume_intr} Let $Y \subset X$ be an $m$-dimensional connected complex submanifold and $(L,h)$ be a Hermitian big line bundle  with curvature form $\theta \coloneqq \frac{\mathrm{i}}{2\pi} \Theta(h)$. Let $u\in \PSH(X,\theta)$ with  $\nu(u,Y) =0$. Then for any holomorphic line bundle $T$ on $X$ we have that
\begin{equation}\label{eq:restr_volume_intr_1}
    \int_Y \left(\theta|_Y+\ddc \Rest_Y^\theta(u)\right)^m = \lim_{k \to \infty} \frac{m!}{k^m} \dim_\mathbb C H^0(Y,\mathcal{O}_Y(L)^k\otimes \mathcal{O}_Y(T) \otimes \mathcal I(ku)|_Y). 
\end{equation}
In case $\theta_u$ is a K\"ahler current, we also have 
\begin{flalign}\label{eq:restr_volume_intr_2}
    \int_Y \left(\theta|_Y+\ddc \Rest_Y^\theta(u)\right)^m  &= \vol_{X|Y}(L,u)\\
    &=\lim_{k \to \infty} \frac{m!}{k^m} \dim_{\mathbb C}\left\{ s|_Y:  s\in H^0(X,\mathcal{O}_X(L)^k \otimes \mathcal{O}_X(T) \otimes \mathcal I(k u))\right\}. \nonumber
\end{flalign}
\end{theorem}

Due to \eqref{eq:restr_volume_intr_1} and \eqref{eq:restr_volume_intr_2}, as well as \eqref{eq: trace_vol_rest_volume_intr} below, the quantity $\int_Y \left(\theta|_Y+\ddc \Rest_Y^\theta(u)\right)^m$ should be perceived as the transcendental restricted volume $\vol_{X|Y}(\{\theta\},u)$.

The K\"ahler current condition for $\theta_u$ can not be dropped in \eqref{eq:restr_volume_intr_2}. Indeed, in view of \eqref{eq: rest_vol_CT_def_intr} below and the comments following \eqref{eq: Mats_thm}, this formula typically fails to hold for $V_\theta$, the potential with minimal singularity type from $\PSH(X,\theta)$, recalled in \eqref{V_theta_def}.

When $Y=X$, the above result recovers \cite[Theorem~1.1]{DX21}. Also, we bring attention to the recent intriguing paper \cite{CMN23}, where the authors study restricted volumes for sections that vanish to some asymptotic order $\tau_j > 0$ along a configuration of subvarieties $S_j \subset X$. Based on  \eqref{eq:restr_volume_intr_2} and ideas from the proof of Corollary~\ref{cor: rest_volume_2}, we expect that the limit of $H^0_0(X|Y,L^p)/p^m$ in their Theorem~1.2 exists and is equal to the mass $\int_Y (\theta|_Y + \ddc \Tr_Y^\theta(u))^m$, where $u \in \PSH(X,\theta)$ is the supremum of all non-positive $\theta$-psh functions on $X$, with Lelong numbers along the $S_j$ at least equal to $\tau_j$. 

\medskip

Lastly, we discuss applications to the classical restricted volume.

The non-nef locus of a big class $\{\theta\}$ is the locus where the Lelong numbers of $V_\theta$ are positive. By Siu's semicontinuity theorem \cite{Siu74}, this set is a countable union of analytic sets.
If $Y \subset X$ is contained in the non-nef locus, one simply defines the (transcendental) restricted volume of $\{\theta\}$ to $Y$ as $\vol_{X|Y}(\{\theta\}) \coloneqq 0$. Otherwise, one employs the following formula  \cite[p.2]{CT22}, that builds on similar notions from \cite{Bo02, Mat13, His12}:
\begin{equation}\label{eq: rest_vol_CT_def_intr}
        \vol_{X|Y}(\{\theta\})\coloneqq \lim_{\varepsilon \searrow 0}\sup_{\varphi} \int_Y (\theta|_Y + \varepsilon \omega|_Y +\ddc \varphi|_Y)^m,
\end{equation} 
where the supremum runs over $\varphi \in \PSH(X,\theta + \varepsilon \omega)$ with analytic singularity type and $\varphi|_Y \not \equiv -\infty$.

The staring point for our inquiry is  Theorem~\ref{thm: rest_volume_tr_volume}, connecting the transcendental restricted volume and the trace, for $Y$ smooth, not contained in the non-nef locus of $\{\theta\}$:
\begin{equation}\label{eq: trace_vol_rest_volume_intr}
\vol_{X|Y}(\{\theta\}) = \int_Y \left(\theta|_Y+\ddc \Rest_Y^\theta(V_\theta)\right)^m.
\end{equation}

If we take $u= V_\theta$ in \eqref{rest_vol_intr} we recover the (algebraic) restricted volume of $L$ of \cite{Ts06, ELMNP, Mat13}:
\begin{flalign}
\label{classical_rest_vol_def_intr}
\vol_{X|Y}(L) &\coloneqq \varlimsup_{k\to\infty} \frac{m!}{k^m} \dim_{\mathbb C}\left\{ s|_Y:  s\in H^0(X,\mathcal{O}_X(L)^k  \otimes \mathcal I(k V_\theta))\right\} \nonumber \\
&=\varlimsup_{k\to\infty} \frac{m!}{k^m} \dim_{\mathbb C}\left\{ s|_Y:  s\in H^0(X,\mathcal{O}_X(L)^k)\right\}. 
\end{flalign}

It is shown in \cite[Corollary~2.15]{ELMNP} that limsup in the above definition is actually a limit. Moreover, when $Y$ is not contained in the augmented base locus of $L$, it is pointed out in \cite[Theorem~1.3]{Mat13}  that the transcendental and algebraic restricted volumes agree (c.f. \cite[Theorem~1.3]{His12}): 
\begin{equation}\label{eq: Mats_thm}
\vol_{X|Y}(\{\theta\}) = \vol_{X|Y}(L).
\end{equation}
We give a novel proof for this identity using the trace operator in Corollary~\ref{cor: rest_volume_2}. See Section~\ref{sec:notation}, for the definition of the augmented base locus/non-K\"ahler locus.

Due to  \cite[Example~2.5]{CPW18} and \cite[Example~2.11]{CT22}, this identity does not hold if $Y$ is contained in the augmented base locus of $L$. As  result, it remains desirable to recover $\vol_{X|Y}(\{\theta\})$ using algebraic information of the line bundle $L$, for arbitrary smooth $Y$. We accomplish this in our last corollary, that is simply a concatenation of \eqref{eq:restr_volume_intr_1} for $u= V_\theta$ and \eqref{eq: trace_vol_rest_volume_intr}:
\begin{corollary}[{\cref{thm: rest_volume_tr_volume} and \cref{cor: rest_volume_2}}]\label{cor: rest_volume_intr} Let $(L,h)$ be a Hermitian big line bundle on $X$ with curvature form $\theta: = \frac{i}{2\pi} \Theta(h)$. Let $Y \subset X$ be an $m$-dimensional connected complex submanifold not contained in the non-nef locus of $L$. Then
\begin{flalign}\label{eq:restr_volume_cor_intr_1}
\int_Y \left(\theta|_Y+\ddc \Rest_Y^\theta(V_\theta)\right)^m& = \vol_{X|Y}(\{\theta\}) \\
& =  \lim_{k \to \infty} \frac{m!}{k^m} h^0(Y,\mathcal{O}_Y(L)^k\otimes \mathcal{O}_Y(T) \otimes \mathcal I(kV_{\theta})|_Y). \nonumber
\end{flalign}
The second equality holds for $Y$ contained in the non-nef locus as well, in this case both quantities involved are just zero.
\end{corollary}

For $Y$ smooth, this corollary strengthens  \cite[Theorem~2.13]{ELMNP} and \cite[Proposition~4.3]{Mat13} in two ways. First, we note that their works only consider \eqref{eq:restr_volume_cor_intr_1} in case of $Y$ that is not contained in the augmented base locus/non-K\"ahler locus.
Second, the above-mentioned results only consider limsup on the right-hand side of \eqref{eq:restr_volume_cor_intr_1}, and the existence of this limit appears to be novel. See \cite[Theorem~1.3]{CT22} for a related result that involves global sections and twisting by ample bundles.

Involving embedded resolution of singularities (as surveyed in \cite{Wlo09}), we expect  that \eqref{eq:restr_volume_cor_intr_1} can be eventually shown for a (singular) subvariety $Y \subset X$ as well. We avoid this direction here, as it would require defining the trace operator in this more general setting (see \cite[Section~2.3]{XiaNA}).

\paragraph{$L^2$ extension theorems.} \cref{thm: rest_volume_intr} will be obtained using techniques of \cite{DX22}, and  $L^2$  extension theorems for the trace operators. Such results are reminiscent of the classical Ohsawa--Takegoshi theorem \cite{OT87}, but they lack an $L^2$ bound.

By its very nature, the trace operator facilitates the following sheaf-theoretic $L^2$ extension theorem:

\begin{theorem}[{\cref{thm:OT}}]\label{thm: OT_intr} Suppose that $u\in \PSH(X,\theta)$ satisfies $\nu(u,Y)=0$, where $Y \subset X$ is a connected complex submanifold. For any $\lambda >0$ we have
\[
\mathcal I(\lambda \Rest_Y^\theta(u)) \subseteq \mathcal I(\lambda u)|_Y.
\]
\end{theorem} 
Here $\mathcal I(\lambda u)|_Y$ simply means the sheaf of holomorphic functions on $Y$ that have local extensions inside  $\mathcal I(\lambda u)$ (see the discussion following \cref{thm:OT}).

In the particularly mysterious case when $u|_Y \equiv -\infty$ but $\nu(u,Y)=0$, this result yields a completely novel $L^2$ extension theorem. Moreover, this theorem refines the sheaf-theoretic part of the classical Ohsawa--Takegoshi extension theorem, when $u|_Y \not \equiv -\infty$ \cite{OT87}. Indeed, in this case  we have $u|_Y \leq \Rest_Y^\theta(u)+C$ for some constant $C$ (Remark~\ref{rem:naiverestprecRest}), implying that
\[
\mathcal I(\lambda u|_Y) \subseteq \mathcal I(\lambda \Rest_Y^\theta(u)) \subseteq \mathcal I(\lambda u)|_Y.
\]

The next global $L^2$ extension theorem also plays an important role in our arguments:

\begin{theorem}[{\cref{thm: OT_ext_global}}]\label{thm: OT_ext_global_intr} Let $Y \subset X$ be a connected complex submanifold. Take a big Hermitian line bundle $(L,h)$ with curvature form $\theta \coloneqq \frac{\mathrm{i}}{2\pi} \Theta(h)$, and a holomorphic line bundle $T$ on $X$. 

Suppose $u\in \PSH(X,\theta)$, $\theta_u$ is a K\"ahler current and $\nu(u,Y)=0$. Then there exists $k_0$ such that for all $k \geq k_0$ and  $s \in H^0(Y, \mathcal{O}_Y(L)^k\otimes \mathcal{O}_Y(T) \otimes \mathcal I(k \Rest_Y^\theta(u)))$, there exists an extension $\tilde s \in H^0(X, \mathcal{O}_X(L)^k \otimes \mathcal{O}_X(T)  \otimes \mathcal I(ku))$ of $s$.
\end{theorem}

While not needed for our purposes, it remains an interesting question whether the $L^2$-norm of $\tilde{s}$ can be controlled, as in the more classical versions of the Ohsawa--Takegoshi theorem \cite{OT87, Bl_sui, GZ_ann}.

The K\"ahler current condition for $\theta_u$ can not be dropped in this result. Indeed, this theorem typically fails to hold for $V_\theta$. Suppose in contrary that the result held for $V_\theta$, then the short argument of Theorem~\ref{thm: rest_volume_2} would also hold for $V_\theta$, yielding \eqref{eq:restr_volume_intr_2} for $V_\theta$. However this is false, per the comments after Theorem~\ref{thm: relativeDX}.

\paragraph{Further directions and applications.} 
 When $\int_Y \left(\theta|_Y+\ddc \Rest^\theta_Y(u)\right)^m=0$, the trace operator defined in \eqref{eq: Tr_general_def_intr} exhibits pathological behaviors. To make this concept more robust, one could perturb $\theta$ to $\theta+\omega$. Of course the  function $\Rest^{\theta+\omega}_Y(u)$ depends on the choice of $\omega$, but one can show that its $\mathcal{I}$-equivalence class is uniquely determined by the current $\theta_u$ (see \cref{lem: trace_vol_prop}(iv)). Using this observation, the trace operator has an intriguing application in non-Archimedean pluripotential theory. Boucksom--Jonsson \cite{BJ18b} developed a global pluripotential theory over the Berkovich analytification of $X$ (with respect to the trivial valuation on $\mathbb{C}$). In \cite{DXZ23}, we introduced a transcendental approach to Boucksom--Jonsson's theory using the theory of test curves.  
If we apply the trace operator pointwise to the test curve representation of a non-Archimedean metric $\phi$, we get the restriction of $\phi$ to the Berkovich analytification of $Y$.
This result is explained in detail in \cite[Theorem~4.21]{XiaNA}, and we refer the reader to this companion paper for more details.

As in \cite[Theorem~1.2]{DX21},  \cref{thm: rest_volume_intr} should allow to prove a weak convergence result about quantization of equilibrium measures. Results of similar nature have been obtained in \cite{BBWN11, CMN23,His12,Vuajm}.

The trace operator is closely linked to transcendental Okounkov bodies, as studied in \cite{Deng17, DRWNXZ}. It allows one to define the valuation vector of a closed positive $(1,1)$-current along a smooth flag, leading to an alternative description of transcendental Okounkov bodies of a big cohomology class. This definition agrees with the one in \cite{Deng17}; see \cite[Section~10.3.3]{Xiabook} for details.

Recently, the theory of trace operators has been applied to the study of b-divisors, yielding some novel estimates of the volumes of currents. We refer to \cite[Theorem~C]{Xiabdiv2}.

Understanding to what extent can potentials dominated by $\Tr_Y^\theta(V_\theta)$ be extended from $Y$ to a $\theta$-psh functions on $X$, would likely lead to progress on the null-locus = non-K\"ahler locus conjecture of Collins--Tosatti (see \cite[Conjecture 1.1]{CT22}). Most likely, this will require a much deeper understanding of extension phenomenon in K\"ahler geometry (c.f. \cite{Bl_sui,GZ_ann,CGZ13, CT14,DGZ16,DRWNXZ}).

\paragraph{Acknowledgments.} We benefited from discussions with Sébastien Boucksom, Junyan Cao, Huayi Chen, Siarhei Finski, Vincent Guedj, Henri Guenancia, Chen Jiang, Rémi Reboulet,  Yi Yao, David Witt Nystr\"om, Kewei Zhang, Zhiwei Wang and Yu Zhao. We thank Sébastien Boucksom, Shin-ichi Matsumura, and  Valentino Tosatti for their comments on a preliminary draft.

The first-named author is partially supported by an Alfred P. Sloan Fellowship and National Science Foundation grants DMS–1846942,DMS-2405274.
The second-named author is supported by the Knut och Alice Wallenbergs Stiftelse grant KAW 2021.0231.

\section{Notations and preliminary results}\label{sec:notation}
In the whole paper, $\ddc$ is defined as $\frac{\mathrm{i}}{2\pi}\partial \overline{\partial}$. All Monge--Amp\`ere type products are taken in the non-pluripolar sense of \cite{BEGZ10}.

Let $X$ be a connected compact K\"ahler manifold of dimension $n$ and $\theta$ be a smooth closed real $(1,1)$-form on $X$ representing a pseudoeffective cohomology class $\{\theta\}$. 

By $\QPSH(X)$ we denote the set of quasi-plurisubharmonic functions (qpsh) on $X$. We denote by $\PSH(X,\theta)$ the set of $\theta$-plurisubharmonic  ($\theta$-psh) functions on $X$. 
We write
\begin{equation}\label{V_theta_def}
V_{\theta}=\sup\{u\in \PSH(X,\theta): u\leq 0\}\in \PSH(X,\theta).
\end{equation}
By $\PSH_>(X,\theta) \subset \PSH(X,\theta)$ we denote the space of potentials $u$ such that $\theta_u\coloneqq \theta+\ddc u$ is a K\"ahler current. We assume that $\{\theta\}$ is big, i.e.,  $\PSH_{>}(X,\theta)\neq \varnothing$.
\paragraph{K\"ahler locus, nef locus, augmented base locus.}We say that $x \in X$ is contained in the K\"ahler locus of $\{\theta \}$ if there exists $u \in \PSH_>(X,\theta)$ that is smooth in a neighborhood of $x$. As shown in \cite[Théorème~2.1.20(ii)]{Bou02}, the K\"ahler locus is a Zariski open set, i.e., the non-K\"ahler locus $\mathrm{nK}(\{\theta\}) \subset X$ is an analytic set. When $\{\theta\}$ is the first Chern class of a big line bundle $L$, then the non-K\"ahler locus is the same as the so called augmented base locus of $L$. We refer to  \cite[Corollaire~2.2.8]{Bou02} for this result, as well as the definition of the augmented base locus.

The non-nef locus $\mathrm{nn}(\{\theta\}) \subset X$ is the set where the Lelong numbers of $V_\theta$ are positive. Due to Siu's semicontinuity theorem \cite{Siu74}, $\mathrm{nn}(\{\theta\})$ is a countable union of analytic sets and $\mathrm{nn}(\{\theta\}) \subseteq \mathrm{nK}(\{\theta\})$.
\paragraph{Singularity types and envelopes.} For $u,v\in \QPSH(X)$, we write 
\[
u\preceq v
\]
if $u\leq v+C$ for some $C\in \mathbb{R}$. The corresponding equivalence relation is denoted by $\sim$, and the equivalence classes are the singularity types $[u]$.

Following \cite{KS20} and \cite[Section~2.4]{DX22}, it is possible to refine $\preceq$ in the following manner: given $u,v\in \QPSH(X)$, we write that 
\[
u \preceq_\mathcal I v
\]
if for all  $\lambda >0$ we have that the inclusion of multiplier ideal sheaves $\mathcal I(\lambda u) \subset \mathcal I(\lambda v)$. The corresponding notion of equivalence relation will be denoted by $\sim_\mathcal I$, and the equivalence classes are the $\mathcal I$-singularity types $[u]_\mathcal I$. 

As proved in \cite[Corollary~2.16]{DX22} (rephrasing  \cite{BFJ08}), $u\preceq_{\mathcal{I}}v$ if and only if for any prime divisor $E$ over $X$, we have
$\nu(u,E)\geq \nu(v,E)$. Here $\nu(u,E)$ is the generic Lelong number of $u$ along $E$. More precisely, take a proper smooth bimeromorphic model $\pi\colon Y\rightarrow X$ such that $E$ is a prime divisor on $Y$. Then $\nu(u,E)$ is the infimum of the Lelong numbers $\nu(\pi^*u,x)$ with $x$ running over $E$. 

Given $u \in \PSH(X,\theta)$, one can define the following notions of envelopes:
\[
\begin{aligned}
P_\theta[u] \coloneqq & \sup\{v \in \PSH(X,\theta) :  v\leq 0, v \preceq u  \} \in \PSH(X,\theta),\\
P_\theta[u]_\mathcal I \coloneqq & \sup\{v \in \PSH(X,\theta): v \leq 0 , v \preceq_\mathcal I u\} \in \PSH(X,\theta).
\end{aligned}
\]
These envelopes were studied in depth in \cite{RWN14,DDNL18mono, DX22, Tr22}. Recall that $\int_X (\theta+\ddc P_\theta[u])^n=\int_X \theta_u^n$, while  $\int_X (\theta+\ddc P_\theta[u]_\mathcal I)^n \geq \int_X \theta_u^n$.

Given $u \in \PSH(X,\theta)$, we say that $u$ is model (or more precisely $\theta$-model) if $u = P_\theta[u]$. Similarly, we say that $u$ is $\mathcal I$-model  if $u = P_\theta[u]_\mathcal I$.

\begin{proposition}\label{prop: vol_limit_model}
Suppose that $u \in \PSH(X,\theta)$ and $u_j \in \PSH(X,\theta + \varepsilon_j \omega)$ for some $\varepsilon_j \searrow 0$.
If $u_j \searrow u$ and $u_j = P_{\theta + \varepsilon_j \omega}[u_j]$ for all $j \geq 0$ then
\[
\lim_{j\to\infty}\int_X (\theta + \varepsilon_j \omega)^n_{u_j}=\int_X \theta_u^n.
\]
Moreover, if $\int_X \theta_u^n>0$, then for any prime divisor $E$ over $X$, we have
\[
\lim_{j\to\infty} \nu(u_j,E)=\nu(u,E).
\]
\end{proposition}

\begin{proof} Since $u_j \geq u$, using \cite[Theorem~1.1]{WN19} we get the following:
\[
\varliminf_{j\to\infty} \int_X (\theta + \varepsilon_j \omega)^n_{u_j} \geq  \varliminf_{j\to\infty} \int_X (\theta + \varepsilon_j \omega)^n_{u}  = \int_X \theta_{u}^n\,.
\]
To finish the proof, we will argue that $\varlimsup_j \int_{X} (\theta + \varepsilon_j \omega)^n_{u_j} \leq \int_{X} \theta^n_{u}$. Indeed, fixing $j_0 \in \mathbb N$, we have
\begin{flalign*}
\varlimsup_{j \to \infty} \int_X (\theta + \varepsilon_j \omega)^n_{u_j}&=\varlimsup_{j \to \infty} \int_{\left\{u_j =0\right\}} (\theta + \varepsilon_j \omega)^n_{u_j} \\
&\leq \varlimsup_{j \to \infty} \int_{\left\{u_j =0\right\}} (\theta + \varepsilon_{j_0} \omega)^n_{u_j}\\
&\leq \int_{\left\{u =0\right\}} (\theta + \varepsilon_{j_0} \omega)^n_{u}\,,
\end{flalign*}
where in the first line we used the condition $u_j = P_{\theta + \varepsilon_j \omega}[u_j]$ and \cite[Theorem~3.8]{DDNL18mono}, and in the last line we have used $u_j \searrow u$ and \cite[Proposition~4.6]{DDNLmetric} (see also \cite[Lemma~2.11]{DDNLsurv}). Letting $j_0 \to \infty$, we arrive at the desired conclusion: 
\[
\varlimsup_{j \to \infty} \int_{X} (\theta + \varepsilon_j \omega)^n_{u_j} \leq \varliminf_{j_0 \to \infty}\int_{\{u =0\}} (\theta + \varepsilon_{j_0} \omega)^n_{u}  =  \int_{\{u =0\}} \theta^n_{u} \leq \int_{X} \theta^n_{u}.
\]

Regarding the part involving Lelong numbers, by \cite[Lemma~4.3]{DDNLmetric} and the proceeding result, for  $j$ big enough there exists $v_j \in \PSH(X,\theta+\varepsilon_j\omega)$ and $\varepsilon_j \searrow 0$ such that $(1-\varepsilon_j) u_j + \varepsilon_j v_j \leq u$. Due to additivity of Lelong numbers, for $j$ big enough we have 
\[
(1-\varepsilon_j)\nu(u_j,E) + \varepsilon_j \nu(v_j,E) \geq \nu(u,E) \geq \nu(u_j,E).
\]

However, for fixed $E$, the Lelong number $\nu(\chi,E)$ is uniformly bounded for any $\chi \in \PSH(X,\theta+\omega)$. Indeed, $\nu(\chi,E)$ is upper semicontinuous with respect to the $L^1$-topology of $\chi$ (see \cite[Exercise~2.7]{GZ17}), while the set of $v\in\PSH(X,\theta+\omega)$ with $\sup_X v=0$ is $L^1$-compact. So letting $\varepsilon \searrow 0$ we conclude that $\nu(u_j,E) \to \nu(u,E)$. 
\end{proof}

\paragraph{Quasi-equisingular approximations.}
We say a quasi-plurisubharmonic function $\varphi$ has analytic singularity type (notation: $[\varphi] \in \mathcal A$) if locally $\varphi$ can be written as
\begin{equation}\label{eq: analyt_sing_type_def}
c\log \sum_{i=1}^N |f_i|^2+g,
\end{equation}
where $c\in \mathbb{Q}_{\geq 0}$, $f_1,\ldots,f_N$ is a finite collection of holomorphic functions and $g$ is bounded. If $g$ is smooth, then $\varphi$ is said to have neat analytic singularity type \cite{Dem16ext}. We refer to \cite[Remark 2.7]{DRWNXZ} for a comparison between the different definitions of analytic singularity type.

We recall the following result of Demailly \cite{Dem92} about the usefulness of potentials with analytic singularity type, implicit in the proof of \cite[Theorem~2.2.1]{DPS01}, \cite[Theorem~3.2]{DP04} and \cite[Theorem~1.6]{Dem15}:

\begin{theorem}\label{thm:Demailly} Let $u \in \PSH(X,\theta)$.  Then there
exists $u_k^D \in \PSH(X,\theta + \varepsilon_k \omega)$ with $\varepsilon_k \searrow 0$ such that
\begin{enumerate}[label=(\roman*)]
\item $u^D_k \searrow u$; 
\item $u^D_k$ has analytic singularity type;
\item $\mathcal I( \frac{s{2^k}}{2^k-s} u_k^D) \subseteq \mathcal I(s u) \subseteq \mathcal I(s u_k^D)$ for all $s >0$.
\item For all $x \in X$  there exists generators $f_1,\ldots f_k$ of $\mathcal I(2^k u)$ in a neighborhood $U$ of $x$ such that
\[
u^D_k|_U = \frac{1}{2^k} \log \sum_{j=1}^k |f_j|^2 + g,
\]
for some $g$, bounded on $U$.
\end{enumerate}
\end{theorem}

For a detailed exposition of the construction of the sequence $u^D_k$ we refer to \cite[Theorem~13.12]{Dem12} and its proof.

\begin{proof}
Parts (i) and (ii) follow from \cite[Theorem~1.6]{Dem15}. The second inclusion of (iii) follows from $u \leq u^D_k$, whereas the first inclusion of (iii) follows from \cite[Corollary~1.12]{Dem15}.

Part (iv) follows from the proof of \cite[Theorem~13.12]{Dem12}. Indeed, the $u^D_k$ are constructed locally using generators of $\mathcal I(2^k u)$, and then the local approximations are glued together.
\end{proof}

\begin{remark}\label{rmk:Demapp}
If $\theta_u$ is a K\"ahler current satisfying $\theta_u \geq \delta \omega$, then for the quasi-equisingular approximation $u^D_j$ of \cref{thm:Demailly} one can assume that $\theta_{u^D_j} \geq \frac{\delta}{2} \omega$, after possibly discarding a finite number of terms (see \cite[Theorem~13.12]{Dem12} and its proof).
\end{remark}

In this work we will need more general approximating sequences, leading to the following terminology inspired from \cite[Definition 2.3]{Cao14} (c.f. \cite[Definition 4.1.3]{Dem15}):
\begin{definition}\label{def:equising} Let $u \in \PSH(X,\theta)$.  Then $u_k \in \PSH(X,\theta + \varepsilon_k \omega)$ is a quasi-equisingular approximation of $u$ if  $\varepsilon_k \searrow 0$ and
\begin{enumerate}[label=(\roman*)]
\item $u_k \searrow u$; 
\item $u_k$ has analytic singularity type;
\item For any $\lambda'>\lambda>0$ we have that $\mathcal I( \lambda' u_k) \subseteq \mathcal I(\lambda u)$ for all $k$ big enough.
\end{enumerate}
\end{definition}

In case of non-vanishing non-pluripolar mass, quasi-equisingular approximation can be characterized in a variety of ways that will be useful for us later:

\begin{theorem}\label{thm: quasi_eqvi_eqv}Let $u \in \PSH(X,\theta)$ be a potential satisfying $\int_X \theta_u^n>0$. Let $\varepsilon_k\searrow 0$ be a decreasing sequence of non-negative numbers and $u_k \in \PSH(X,\theta + \varepsilon_k \omega)$ ($k\geq 1$) be a decreasing sequence of potentials with analytic singularity type, with limit $u$. Then the following are equivalent:
\begin{enumerate}[label=(\roman*)]
\item  $\{u_k\}_k$ is a quasi-equisingular approximation of $u$; 
\item  $P_{\theta + \varepsilon_k \omega}[u_k]_{\mathcal{I}}\searrow P_{\theta}[u]_{\mathcal{I}}$ as $k\to\infty$;
\item  $\int_X (\theta + \varepsilon_k \omega)^n_{u_k} \searrow \int_X \theta_{P_{\theta}[u]_{\mathcal{I}}}^n$ as $k\to\infty$.
\end{enumerate}
\end{theorem}

\begin{proof}
We first show that (iii)$\implies$(i) holds. By \cite[Lemma~4.3]{DDNLmetric}, for $k$ big enough, there exists $\delta_k\searrow 0$ and $w_k \in \PSH(X,\theta + \varepsilon_k \omega)$ such that
\[
(1-\delta_k) u_k + \delta_k w_k \leq P_\theta[u]_\mathcal I.
\]
We need to argue that for $\lambda' > \lambda >0$ we have $\mathcal I( \lambda' u_k) \subseteq \mathcal I(\lambda u)$ for $k$ big enough. Let $f$ be a germ of  $\mathcal I( \lambda' u_k)$ at $x \in X$ for some $k \geq 1$. Then 
\[
\int_U |f|^2\mathrm{e}^{-\lambda' u_k}\,\omega^n<\infty
\]
for some neighborhood $U$ of $x$. We have the following inequalities:
\begin{flalign*}
\int_U |f|^2 \mathrm{e}^{-\lambda P_{\theta}[u]_{\mathcal{I}}}\,\omega^n & 
\leq \int_U |f|^2 \mathrm{e}^{-\lambda ((1-\delta_k) u_k + \delta_k w_k)}\,\omega^n  \\
& \leq \bigg(\int_U |f|^2 \mathrm{e}^{-\lambda' (1 - \delta_k) u_k}\,\omega^n\bigg)^\frac{1}{p}\bigg(\int_U |f|^2 \mathrm{e}^{-\lambda q \delta_k w_k}\,\omega^n\bigg)^\frac{1}{q}\\
& \leq \bigg(\int_U |f|^2 \mathrm{e}^{-\lambda' u_k}\,\omega^n\bigg)^\frac{1}{p}\bigg(\int_U |f|^2 \mathrm{e}^{-\lambda q \delta_k w_k}\,\omega^n\bigg)^\frac{1}{q},
\end{flalign*}
where we have used Hölder's inequality with exponents $p = \lambda'/\lambda $ and $q =\lambda'/(\lambda'-\lambda )$. 
As in the proof of \cref{prop: vol_limit_model}, the Lelong numbers of $w_k$ are uniformly bounded on $X$. It follows from Skoda's integrability theorem \cite[Theorem~2.50]{GZ17} that $\int_U |f|^2 \mathrm{e}^{-\lambda q \delta_k w_k}\,\omega^n$ is finite for $k$ big enough. It follows that $f\in H^0(U,\mathcal{I}(\lambda u))$, for high enough $k$, proving that $\mathcal I( \lambda' u_k) \subseteq \mathcal I(\lambda u)$ for $k$ big enough.

Now we address the direction (i)$\implies$(ii). 
Given $\lambda'>\lambda>0$, we have that $\mathcal I( \lambda' u_k) \subseteq \mathcal I(\lambda u)$ for $k$ big enough. 
Let $v \coloneqq \lim_k P_{\theta + \varepsilon_k \omega } [u_k]_\mathcal I \geq P_\theta [u]_\mathcal I$. We obtain that 
$\mathcal I( \lambda' v) \subseteq \mathcal I( \lambda' P_{\theta + \varepsilon_k \omega } [u_k]_\mathcal I)= \mathcal I( \lambda' u_k) \subseteq \mathcal I(\lambda u)$.

Letting $\lambda' \searrow \lambda$, Guan--Zhou's strong openness theorem \cite{GZ15} gives that $\mathcal I( \lambda v )  \subseteq \mathcal I(\lambda u)$, giving that $v \preceq_\mathcal I u$. This implies that $v \leq P_\theta[u]_\mathcal I$. Since  $v \geq P_\theta[u]_\mathcal I$, we obtain that $v = P_\theta[u]_\mathcal I$, proving (ii).

Finally, we show that (ii)$\implies$(iii). 
From \cite[Proposition~2.20]{DX22} it follows that 
\[
u_j\sim P_{\theta + \varepsilon_j \omega}[u_j]_{\mathcal{I}} = P_{\theta + \varepsilon_j \omega}[u_j].
\]
Using \cite[Theorem~1.1]{WN19} and \cref{prop: vol_limit_model} we conclude that $\int_{X} (\theta + \varepsilon_j \omega)^n_{u_j} =\int_{X} (\theta + \varepsilon_j \omega)^n_{P_\theta[u_j]_\mathcal I} \searrow \int_{X} \theta^n_{P_{\theta}[u]_{\mathcal{I}}}$, hence (iii) holds.
\end{proof}

Given a quasi-plurisubharmonic function $v$ on $X$ with analytic singularity type, let $\mathcal J(v)$ be the sheaf of holomorphic functions, whose germs $f$ satisfy $|f|^2 \leq C \mathrm{e}^{v} $ for some $C \in \mathbb R_+$. This sheaf has a nice property: 
\begin{lemma}\label{lem: coh_J} Let $v \in \QPSH(X)$ be a potential with analytic singularity type. Then $\mathcal J(v)$ is coherent.
\end{lemma}
\begin{proof} It is enough to argue coherence in a local neighborhood $U \subset X$. 
If $\pi\colon Y\rightarrow X$ is a proper bimeromorphic morphism, we have
\[
\mathcal{J}(v)=\pi_* \mathcal{J}(\pi^*v).
\]
So the coherence of $\mathcal{J}(v)$ follows from that of $\mathcal{J}(\pi^*v)$, due to the Grauert direct image theorem. Hence, after choosing suitable $\pi$, we may assume that $v$ has log singularities along a simple normal crossing $\mathbb{Q}$-divisor $D$ on $X$ (see \cite[Lemma~2.3.19]{MM07}). But it is then clear that $\mathcal J(v) = \mathcal{O}_X(- \lceil D\rceil)$, the latter sheaf being coherent.
\end{proof}

\section{The trace operator of quasi\hyp{}plurisubharmonic functions}

Let $X$ be a connected compact K\"ahler manifold of dimension $n$. Take a closed real smooth $(1,1)$-form $\theta$ on $X$ representing a big cohomology class. Fix a K\"ahler form $\omega$ on $X$.

\paragraph{Definition of the trace operator.}
We fix a connected complex submanifold $Y\subseteq X$ of dimension $m$. With \cref{rmk:Demapp} in hand, first we define the trace operator for $u \in \PSH_{>}(X,\theta)$. Recall that $u \in \PSH_{>}(X,\theta) \subset \PSH(X,\theta)$ if $\theta_u$ is a K\"ahler current.

\begin{definition}\label{def: trace_Kahler}Suppose that $u \in  \PSH_{>}(X,\theta)$ with $\nu(u,Y)=0$. Let $u_j^D \in \PSH(X,\theta)$ be a quasi-equisingular approximation of $u$ given by \cref{thm:Demailly}. Then
\begin{equation}\label{eq: Tr_def}
\Rest^\theta_Y(u) \coloneqq \lim_{j\to\infty} P_{\theta|_Y} [u^D_j|_Y].
\end{equation}
\end{definition}

The condition $\nu(u,Y)=0$  implies that $u^D_j|_Y \not \equiv -\infty$, hence the right-hand side of \eqref{eq: Tr_def} makes sense.

As $\Rest^\theta_Y(u)$ only depends on the singularity types $[u_j^D]$, it immediately follows that $\Rest^\theta_Y(u)$ is independent of the choice of $u^D_j$. For later applications it will be useful to show that any quasi-equisingular approximation $v_j \in \PSH(X,\theta + \varepsilon_j \omega)$ of $u$ (recall Definition~\ref{def:equising}) gives the same notion: 
\begin{equation}\label{eq: Tr_quasi_inv}
\Rest^\theta_Y(u) = \lim_{j\to\infty} P_{\theta|_Y + \varepsilon_j \omega|_Y} \left[v_j|_Y\right].
\end{equation}

Before we establish this, let us make two preliminary observations:

\begin{lemma}\label{lem: non_zero_vol} Let $u\in \PSH_{>}(X,\theta)$ with $\nu(u,Y)=0$.
If $\theta_u \geq \delta \omega$ for some $\delta>0$, then 
\[
\int_Y \left(\theta|_Y+\ddc \Rest^\theta_Y(u)\right)^m \geq \frac{\delta^m}{2^m} \int_Y \omega|_Y^m>0.
\]
\end{lemma}
%When $m=0$, both sides are understood as $1$.

\begin{proof}This follows from \cref{rmk:Demapp}, since $\theta_{u^D_j} \geq 2^{-1}\delta \omega$, after discarding a finite amount of terms. 
Using \cite[Proposition~2.20]{DX22}, we get that 
\[
 \int_Y \left(\theta|_Y+\ddc P[u^D_j|_Y]_{\mathcal{I}}\right)^m =\int_Y \left(\theta|_Y+\ddc u^D_j|_Y\right)^m \geq \frac{\delta^m}{2^m}\int_Y \omega|_Y^m
\]
for $j$ big enough. 
Using \cref{prop: vol_limit_model} the conclusion follows.
\end{proof}

\begin{lemma}\label{lem: rescale_inv}
Let $u\in \PSH_{>}(X,\theta)$ with $\nu(u,Y)=0$.
Suppose that $v_j \in \PSH(X,\theta)$ is a quasi-equisingular approximation of $u$ with $\theta_{v_j} \geq \varepsilon \omega$ and $\theta_{u} \geq \varepsilon \omega$ for some $\varepsilon >0$. Then for any $\varepsilon_j \searrow 0$,
\begin{equation}\label{eq:limP1}
\lim_{j\to\infty} P_{\theta|_Y}[v_j |_Y] = \lim_{j\to\infty} P_{\theta|_Y + \varepsilon_j \omega|_Y}\left[
\left(1-j^{-1}\right)v_j |_Y\right].
\end{equation}
\end{lemma}

\begin{proof}
We may assume that $v_j\leq 0$ for each $j$. Observe that $(1-j^{-1})v_j$ is still a decreasing sequence with limit $u$. Moreover, we have
\[
\theta+(1-j^{-1})\ddc v_j\geq j^{-1}\theta+(1-j^{-1})\epsilon\omega\geq 0
\]
for large enough $j$. In particular, for large $j$, we have
$(1-{1}/{j})v_j \in \PSH(X,\theta+ \varepsilon_j \omega)$
and $(1-{1}/{j})v_j$ is also a quasi-equisingular approximation of $u$. So the quantity inside the limit on the right-hand side of \eqref{eq:limP1} makes sense.

Moreover, observe that 
\[
P_{\theta|_Y}\left[v_j |_Y \right] \leq P_{\theta|_Y+ \varepsilon_j \omega|_Y}\left[\left(1-j^{-1}\right)v_j |_Y\right].
\]
This implies that
\[
 \lim_{j\to\infty} P_{\theta|_Y}\left[v_j |_Y \right] \eqqcolon v \leq  v' \coloneqq \lim_{j\to\infty} P_{\theta|_Y+ \varepsilon_j \omega|_Y}\left[\left(1-j^{-1}\right)v_j |_Y \right].
\]
Both $v$ and $v'$ are model potentials in $\PSH(Y,\theta|_Y)$ due to \cite[Corollary~4.7]{DDNLmetric}. Due to multilinearity of non-pluripolar products, 
\[
\begin{aligned}
&\lim_{j\to\infty} \left( \int_Y \left(\theta|_Y+\ddc P_{\theta|_Y}[v_j |_Y] \right)^m - \int_Y \left(\theta|_Y+ \varepsilon_j \omega|_Y+\ddc P_{\theta|_Y+ \varepsilon_j \omega|_Y}[(1-1/j)v_j |_Y]\right)^m \right)\\
=& \lim_{j\to\infty} \left( \int_Y \left(\theta|_Y+\ddc v_j |_Y \right)^m - \int_Y \left(\theta|_Y+ \varepsilon_j \omega|_Y+(1-1/j)\ddc v_j |_Y\right)^m \right)\\
=& 0.
\end{aligned}
\]
\cref{prop: vol_limit_model} and \cref{lem: non_zero_vol} now imply that
\[
\int_Y (\theta|_Y+\ddc v)^m=\int_Y (\theta|_Y+\ddc v')^m>0.
\]
Since, both potentials $v$ and $v'$ are model potentials in $\PSH(Y,\theta|_Y)$, from \cite[Theorem~1.3]{DDNL18mono} we obtain that $v=v'$, finishing the proof.
\end{proof}

Finally, we prove  \eqref{eq: Tr_quasi_inv}.

\begin{lemma}\label{lem: indep_quasi_Tr}
Let $u\in \PSH_{>}(X,\theta)$ with $\nu(u,Y)=0$.
The definition of $\Rest^\theta_Y(u)$ does not depend on the choice of quasi-equisingular approximation. More precisely, \eqref{eq: Tr_quasi_inv} holds.
\end{lemma}
\begin{proof} 
Due  to \cite[Corollary~4.1.7]{Dem15} it is possible to find increasing sequences $a_j,b_j \in \mathbb N$ with $a_j,b_j \nearrow \infty$ such that
\[
\bigg(1 + \frac{1}{j} \bigg)u_{b_j}^D\preceq v_{a_j} \preceq \bigg(1 - \frac{1}{j} \bigg)u_j^D.
\]
This gives
\begin{equation}\label{eq: Env_ests}
P_{\theta|_Y }\big[\left(1 + j^{-1} \right)u_{b_j}^D|_Y\big]\leq P_{\theta|_Y + \varepsilon_j \omega|_Y}[v_{a_j}|_Y] \leq  P_{\theta|_Y + \varepsilon_j \omega|_Y}\left[\left(1 - j^{-1} \right)u_j^D\right].
\end{equation}
Using \cite[Theorem~1.1]{WN19} we conclude that 
\begin{flalign*}
\int_Y \left(\theta|_Y+\ddc (1 + j^{-1} ) u_{b_j}^D|_Y \right)^m 
=& \int_Y \left(\theta|_Y+\ddc P_{\theta|_Y }\big[(1 + j^{-1} )u_{b_j}^D|_Y\big]\right)^m\\
\leq & \int_Y \left(\theta|_Y + \varepsilon_j \omega|_Y+\ddc P_{\theta|_Y + \varepsilon_j \omega|_Y}[v_{a_j}|_Y]\right)^m \\
\leq &  \int_Y \left(\theta|_Y + \varepsilon_j \omega|_Y+\ddc P_{\theta|_Y + \varepsilon_j \omega|_Y}\left[(1 - j^{-1} )u_j^D|_Y\right]\right)^m\\
 = & \int_Y \left(\theta|_Y + \varepsilon_j \omega|_Y+\ddc (1 - j^{-1} )u_j^D|_Y\right)^m.
\end{flalign*}

Due to multilinearity of non-pluripolar products, the volume at the beginning and the end converges to $\int_Y (\theta|_Y+\ddc \Rest^\theta_Y(u))^m>0$, hence all the volumes do the same.

Let $w \coloneqq \lim_j P_{\theta|_Y + \varepsilon_j \omega|_Y}[v_{a_j}|_Y]=\lim_j P_{\theta|_Y + \varepsilon_j \omega|_Y}[v_{j}|_Y] \in \PSH(Y,\theta|_Y)$. By \eqref{eq: Env_ests} and \cref{lem: rescale_inv} we obtain that $\Rest^\theta_Y(u) \geq w$. Since both these potentials are model in $\PSH(Y,\theta|_Y)$ with the same non-zero volume, due to \cite[Theorem~1.3]{DDNL18mono} we obtain that $\Rest^\theta_Y(u) = w$, finishing the proof.
\end{proof}

From the above lemma (and its proof), it also follows that $\Rest^\theta_Y(u)= \lim_{\varepsilon \to 0} \Rest^{\theta+ \varepsilon \omega}_Y(u)$
for any $u \in \PSH_>(X,\theta)$ with $\nu(u,Y)=0$, opening the door to the following general definition of trace, consistent with Definition~\ref{def: trace_Kahler}:

\begin{definition}\label{def: trace_general}
For any $ u \in \PSH(X,\theta)$ with $\nu(u,Y)=0$, we define
\[
\Rest^\theta_Y(u) \coloneqq \lim_{\varepsilon \to 0} \Rest^{\theta+ \varepsilon \omega}_Y(u).
\]
It is immediate that this definition is independent of the choice of $\omega$.
\end{definition}
\begin{example}\label{ex:traceYequalX}
    When $Y=X$, for any $u\in \PSH(X,\theta)$  we have
    \[
    \Tr_X^{\theta}(u)=P_{\theta}[u]_{\mathcal{I}}.
    \]
    In fact, for each $\epsilon>0$, we have $\Tr_X^{\theta+\epsilon\omega}(u)=P_{\theta+\epsilon\omega}[u]_{\mathcal{I}}$
    by \cite[Proposition~3.3]{DX21}. The formula above then follows from the fact that $P_{\theta+\epsilon\omega}[u]_{\mathcal{I}} \searrow P_{\theta}[u]_{\mathcal{I}}$, since the limit on the left-hand side easily seen to be a candidate for $P_{\theta}[u]_{\mathcal{I}}$.
\end{example}
\begin{example}\label{ex:analytic}
    Assume that $u\in \PSH(X,\theta)$ has analytic singularity type and $\nu(u,Y)=0$. Then
    \[
    \Tr_Y^{\theta}(u)=P_{\theta|_Y}[u|_Y]_{\mathcal{I}}=P_{\theta|_Y}[u|_Y] \simeq [u|_Y].
    \]
In fact, by \cref{lem: indep_quasi_Tr}, we have $\Tr_Y^{\theta+\epsilon\omega}(u)=P_{\theta|_Y+\epsilon\omega|_Y}[u|_Y]_{\mathcal{I}}=P_{\theta|_Y+\epsilon\omega|_Y}[u|_Y] \simeq [u|_Y]$ for all $\varepsilon>0$. 
As $P_{\theta|_Y+\epsilon\omega|_Y}[u|_Y]_{\mathcal{I}} 
 \simeq [u|_Y] $, we get that $P_{\theta|_Y+\epsilon\omega|_Y}[u|_Y]_{\mathcal{I}} \searrow P_{\theta|_Y}[u|_Y]_{\mathcal{I}}$. 
\end{example}

\begin{remark}\label{rem:naiverestprecRest}
    Let $u\in \PSH(X,\theta)$. If $u|_Y\not\equiv -\infty$, then $u|_Y\preceq \Rest_Y^\theta(u)$. Indeed, using  \cref{def: trace_general}, it suffices to prove that  $u|_Y - \sup_Y u \leq  \Rest_Y^{\theta + \varepsilon \omega}(u)$ for all $\varepsilon>0$. This assertion follows easily from \cref{def: trace_Kahler}. 
\end{remark}

\paragraph{Properties of the trace operator.}
We list some immediate properties of the trace operator:

\begin{lemma}\label{lem: trace_vol_prop} 
\leavevmode
\begin{enumerate}[label=(\roman*)]
\item Let $v,w \in \PSH(X,\theta)$, with $\nu(v,Y)=\nu(w,Y)=0$ and $v \preceq_\mathcal I w$. Then $\Rest^\theta_Y(v) \leq \Rest^\theta_Y(w)$.
In particular, if $v\simeq_{\mathcal{I}} w$, then $\Rest^\theta_Y(v) = \Rest^\theta_Y(w)$. 
\item Let $ u \in \PSH(X,\theta)$ with $\nu(u,Y)=0$.  Then $\Rest^\theta_Y(u) = P_{\theta|_Y}[\Rest^\theta_Y(u)]$ and 
\[
\int_Y \left(\theta|_Y + \varepsilon \omega|_Y+\ddc \Rest^{\theta+ \varepsilon \omega}_Y(u)\right)^m 
\searrow \int_Y \left(\theta|_Y+\ddc \Rest^\theta_Y(u)\right)^m  \quad \text{ as } \varepsilon \searrow 0.
\]
\item Let $u\in \PSH(X,\theta)$ with $\nu(u,Y)=0$ and $\theta'=\theta+\ddc g$ for some smooth real function $g$ on $X$. Write $u'=u-g\in \PSH(X,\theta')$. If  $\int_Y(\theta|_Y+\ddc \Tr^{\theta}_Y(u))^m>0$,  then $\Tr^{\theta}_Y(u)\sim_{\mathcal{I}} \Tr^{\theta'}_Y(u')$. 
\item Let $u\in \PSH(X,\theta)$ with $\nu(u,Y)=0$. If  $\int_Y(\theta|_Y+\ddc \Tr^{\theta}_Y(u))^m>0$, then we have  $\Tr_Y^{\theta}(u)\sim_{\mathcal{I}} \Tr_Y^{\theta+\omega}(u).$
\end{enumerate}
\end{lemma}

\begin{proof} (i) Assume that $v\preceq_{\mathcal{I}}w$.  The fact that $v^D_j \preceq w^D_j$ for the quasi-equisingular approximations of $v,w \in \PSH(X, \theta + \varepsilon \omega)$ comes from the construction of \cref{thm:Demailly}. As a result, $\Rest^{\theta + \varepsilon \omega}_Y(v) \leq \Rest^{\theta + \varepsilon \omega}_Y(w)$. Let $\varepsilon \to 0$, the result follows.

\smallskip
\noindent (ii) Notice that $\Rest^{\theta+ \varepsilon \omega}_Y(u)$ are model potentials in $\PSH(Y,\theta|_Y+ \varepsilon \omega|_Y)$ due to \cite[Corollary~4.7]{DDNLmetric}. We argue that $\Rest^{\theta}_Y(u)\coloneqq  \lim_{\varepsilon \searrow }\Rest^{\theta+ \varepsilon \omega}_Y(u)$ is also a model potential in $\PSH(Y,\theta|_Y)$. Let $v \in \PSH(Y,\theta|_Y)$ with $v \leq 0$ and $v \preceq \Rest^{\theta}_Y(u)$. Then also $v \preceq \Rest^{\theta + \varepsilon \omega}_Y(u)$, which implies that $v \leq \Rest^{\theta + \varepsilon \omega}_Y(u)$. Taking the limit, it results that $v \leq \Rest^{\theta}_Y(u)$, showing that $\Rest^{\theta}_Y(u)$ is $\theta|_Y$-model.

The convergence of volumes follows from \cref{prop: vol_limit_model}.

\smallskip
\noindent (iii) Let $E$ be a prime divisor over $Y$. It suffices to show that 
\[
\nu(\Tr^{\theta}_Y(u),E)=\nu(\Tr^{\theta'}_Y(u'),E).
\]
By \cref{def: trace_general} and \cref{prop: vol_limit_model}, it suffices to show 
\[
\nu(\Tr^{\theta+\epsilon \omega}_Y(u),E)=\nu(\Tr^{\theta'+\epsilon\omega}_Y(u'),E)
\]
for any $\epsilon>0$. Let $u_j\in \PSH(X,\theta+\epsilon\omega)$ be a quasi-equisingular approximation of $u$, then $u_j-g$ is a quasi-equisingular approximation of $u'$. From \cref{prop: vol_limit_model} it follows that
\[
\nu(\Tr^{\theta+\epsilon \omega}_Y(u),E)=\lim_{j\to\infty}\nu(u_j|_Y,E)=\lim_{j\to\infty}\nu(u_j|_Y-g|_Y,E)=\nu(\Tr^{\theta'+\epsilon\omega}_Y(u'),E).
\]

\noindent (iv) Let $E$ be a prime divisor over $Y$. It suffices to show that
\[
\nu(\Tr_Y^{\theta}(u),E)= \nu(\Tr_Y^{\theta+\omega}(u),E).
\]
In view of \cref{prop: vol_limit_model} and \cref{def: trace_general}, it suffices to show that
\[
\nu(\Tr_Y^{\theta+\epsilon\omega}(u),E)= \nu(\Tr_Y^{\theta+\omega}(u),E)
\]
for each $\epsilon>0$. Let $u_j\in \PSH(X,\theta+\epsilon\omega)$ be a quasi-equisingular approximation of $u$, viewed as an element in $\PSH(X,\theta+\epsilon\omega)$. Then $u_j$ is also a quasi-equisingular approximation of $u$, viewed as an element in $\PSH(X,\theta+\omega)$. It follows that 
\[
\nu(\Tr_Y^{\theta+\epsilon\omega}(u),E)=\lim_{j\to\infty}\nu(u_j|_Y,E)= \nu(\Tr_Y^{\theta+\omega}(u),E).
\]
\end{proof}

\begin{lemma}\label{lem: linear_trace} The trace operator satisfies the following linearity properties:
\begin{enumerate}[label=(\roman*)]
    \item Suppose that $u \in \PSH_>(X,\theta)$ and $v \in \PSH_>(X,\theta')$ for another closed real smooth $(1,1)$-form $\theta'$ on $X$. Assume that $\nu(u,Y)=\nu(v,Y)=0$, then
\begin{equation}\label{eq: linear_Tr}
\Rest_Y^\theta(u) + \Rest_Y^{\theta'}(v) \simeq _\mathcal I \Rest_Y^{\theta + \theta'}(u + v).
\end{equation}
\item Let $u \in \PSH_>(X,\theta)$ with $\nu(u,Y)=0$. For any $\lambda>0$, we have
$\lambda\Rest_Y^\theta(u)=\Rest_Y^{\lambda\theta}(\lambda u).$
\end{enumerate}
\end{lemma}
\begin{proof}
    (i) Let $u_j^D\in \PSH(X,\theta)$, $v_j^D\in \PSH(X,\theta')$ be quasi-equisingular approximations of $u$ and $v$ respectively constructed as in \cref{thm:Demailly}.
    It follows from \cite[Theorem~4.1.9]{Dem15} that $u_j^D+v_j^D$ is a quasi-equisingular approximation of $u+v$. 

    We argue \eqref{eq: linear_Tr}. By \cref{lem: non_zero_vol},
    \[
    \int_Y \left( \theta|_Y+\ddc \Rest_Y^\theta(u) \right)^m>0,\quad  \int_Y \left( \theta'|_Y+\ddc \Rest_Y^{\theta'}(v) \right)^m>0,
    \]
 \cref{prop: vol_limit_model} gives that for any divisor $E$ over $Y$ we have:
\[
\begin{split}
\lim_{j\to\infty} \nu(u^D_j|_Y,E) = \nu(\Rest_Y^\theta(u),E), \quad \lim_{j\to\infty} \nu(v^D_j|_Y,E) = \nu(\Rest_Y^\theta(v),E),\\ \lim_{j\to\infty} \nu(u^D_j|_Y+v^D_j|_Y,E) = \nu(\Rest_Y^{\theta+\theta'}(u+v),E).
\end{split}
\]
From \cite[Corollary~2.16]{DX22}, \eqref{eq: linear_Tr} immediately follows.

(ii) Let $u_j\in \PSH(X,\theta)$ be a quasi-equisingular approximation of $u$. We may assume that $u,u_j\leq 0$. Take an increasing sequence of positive rational numbers $\lambda_j$ with limit $\lambda$. It follows that $\lambda_j u_j$ is a quasi-equisingular approximation of $\lambda u$. The same argument as in the previous part shows that 
\[
\lambda\Rest_Y^\theta(u)\sim_{\mathcal{I}}\Rest_Y^{\lambda\theta}(\lambda u).
\]
By since both sides are $\mathcal{I}$-model potentials in $\PSH(Y,\lambda \theta|_Y)$, it follows that they are actually equal.
\end{proof}

Next we prove some limit theorems about the trace.

\begin{lemma}\label{lma:Traceasquasiequiapp}
    Let $u\in \PSH(X,\theta)$ with $\nu(u,Y)=0$. Let $u_j\in \PSH(X,\theta+\epsilon_j\omega)$ be a quasi-equisingular approximation of $u$. Then
    \begin{equation}\label{eq:Trvolappgeneral}
    \lim_{j\to\infty}\int_Y \left(\theta|_Y+\epsilon_j\omega|_Y+\ddc u_j|_Y\right)^m= \int_Y \left(\theta|_Y+\ddc \Tr_Y^{\theta}(u)\right)^m.
    \end{equation}
    In particular, if the right-hand side of \eqref{eq:Trvolappgeneral} is positive, we have
    \begin{equation}\label{eq:TrYgeneral}
    \Tr_Y^{\theta}(u)=\lim_{j\to\infty} P_{\theta|_Y+\epsilon_j\omega|_Y}[u_j|_Y].
    \end{equation}
\end{lemma}
\begin{proof}
Fixing $\varepsilon >0$, after discarding finitely many terms, the sequence $u_j \in \PSH(X,\theta+\epsilon\omega)$ is a quasi-equisingular approximation of $u \in \PSH(X,\theta+\epsilon\omega)$.
Using \cref{lem: indep_quasi_Tr} and \cref{prop: vol_limit_model}
we can start the following sequence of estimates
    \[
    \begin{split}
    \int_Y\left(\theta|_Y+\epsilon \omega|_Y+\ddc \Tr_Y^{\theta+\epsilon\omega}(u)\right)^m & =\lim_{j\to\infty}\int_Y \left(\theta|_Y+\epsilon \omega|_Y+\ddc P_{\theta|_Y+\epsilon \omega|_Y}[u_j|_Y]\right)^m\\
    &=\lim_{j\to\infty}\int_Y \left(\theta|_Y+\epsilon \omega|_Y+\ddc u_j|_Y\right)^m\\
    & \geq  \varlimsup_{j\to\infty}\int_Y \left(\theta|_Y+\epsilon_j\omega|_Y+\ddc u_j|_Y\right)^m\\
    &\geq  \varliminf_{j\to\infty}\int_Y \left(\theta|_Y+\epsilon_j\omega|_Y+\ddc u_j|_Y\right)^m\\
    &\geq \int_Y \left(\theta|_Y+ \ddc \Tr_Y^{\theta}(u) \right)^m,
    \end{split}
    \]
where in the last line we used \cite[Theorem~1.1]{WN19} for $\Tr_Y^{\theta}(u) \preceq u_j|_Y$.
    Let $\epsilon \searrow 0$, \cref{lem: trace_vol_prop}(ii) gives \eqref{eq:Trvolappgeneral} as desired.

It remains to argue \eqref{eq:TrYgeneral}. By \cref{prop: vol_limit_model} and \eqref{eq:Trvolappgeneral}, potentials on both sides of \eqref{eq:TrYgeneral} have the same positive mass. Due to the argument of \cref{lem: trace_vol_prop}(ii)  both sides are also $\theta|_Y$-model potentials, and the left-hand side is clearly dominated by the right-hand side. We conclude that they are actually equal by \cite[Theorem~1.3]{DDNL18mono}.
\end{proof}

We argue continuity of the trace operator along decreasing sequences.

\begin{lemma}\label{lem: trace_decreasing_convergence}Suppose that $u_j,u \in \PSH(X,\theta -\delta \omega)$ for some $\delta >0$. Assume that $P_\theta [u_j]_\mathcal I \searrow P_\theta[u]_\mathcal I$ as $j\to\infty$ and $\nu(u,Y)=0$. Then $\Rest_Y^\theta[u_j] \searrow \Rest_Y^\theta[u]$ as $j\to\infty$.
\end{lemma}

\begin{proof} Let $u^D_{j,k} \in \PSH(X,\theta - \frac{\delta}{2}\omega)$ be the quasi-equisingular approximation of $u_j$, guaranteed by \cref{thm:Demailly}.
As the singularity type of $u^D_{j,k}$ is governed by $(2^{-k},\mathcal I(2^k u_j))$ (see \cref{thm:Demailly}(iv)), we obtain that the singularity type of $u^D_{j,k}$ is monotone-decreasing in both indices.

Due to \cref{prop: vol_limit_model} we have that
\[
\begin{aligned}
\int_X \theta^n_{P_\theta[u^D_{j,k}]} \searrow & \int_X \theta^n_{P_\theta[u_j]_\mathcal I} \geq \int_X \theta^n_{u_j}  \geq \delta^n \int_X \omega^n.\\
\int_X \theta^n_{P_\theta[u_{j}]_\mathcal I} \searrow & \int_X \theta^n_{P_\theta[u]_\mathcal I} \geq \int_X \theta^n_{u}  \geq \delta^n \int_X \omega^n.
\end{aligned}
\]
As the singularity type of $u^D_{j,k}$ is monotone in both indices, using \cite[Theorem~1.1]{WN19} and the above convergences, there exists $k_j \nearrow \infty$ so that
\[
\int_X \theta^n_{P_\theta[u^D_{j,k_j}]_\mathcal I} \searrow \int_X \theta^n_{P_\theta[u]_\mathcal I} \geq \delta^n \int_X \omega^n.
\]
By \cref{thm: quasi_eqvi_eqv} we obtain that $u^D_{j,k_j}$ is a quasi-equisingular approximation of $u$. 
In addition, $P_{\theta|_Y}[u^D_{j,k_j}] \geq \Rest_Y^\theta[u_j] \geq \Rest_Y^\theta[u]$.
The last two sentences and \cref{lem: indep_quasi_Tr} now give that 
$\Rest_Y^\theta[u_j] \searrow \Rest_Y^\theta[u]$, as desired.
\end{proof}

Perhaps unexpectedly, the analogous property does not hold along increasing sequences.
\begin{example}
     Consider $X=\mathbb{P}^2$ and $\theta$ is the Fubini--Study form on $X$. For simplicity, we use affine coordinates $z_1,z_2$ on $\mathbb{C}^2\subseteq \mathbb{P}^2$. Consider $\theta$-psh functions $u_k$ with isolated singularities at $0$:
     \[
     u_k(z_1,z_2)=\max\left\{ \log |z_1|^2,  k^{-1}\log |z_2|^2 \right\}-\log \left(1+|z_1|^2+|z_2|^2\right).
     \]
     The sequence is increasing with usc regularized limit $0$. However, the trace operator of $u_k$ with respect to the hyperplane $\{z_2=0\}$ is independent of $k$ and has a log pole at the point $\{z_1=0\}$.
\end{example}
The failure of the continuity can also be easily understood in the toric situation, where the trace operator corresponds to intersecting the Newton body of the psh metric with the hypersurface corresponding to the submanifold. The latter intersection is clearly not continuous with respect to increasing sequences.
See \cite[Theorem~12.3.2]{Xiabook} for the details.

\paragraph{Restricted volumes and the trace operator.}   When $Y$ is not contained in the non-K\"ahler locus of $\{\theta\}$, Matsumura defines the restricted volume of $\{\theta\}$ to $Y$ in the following manner (\cite[Definition~1.4]{Mat13}):
\begin{equation}\label{eq: rest_vol_Mat_def}
        \vol_{X|Y}(\{\theta\})\coloneqq \sup_{\substack{\varphi\in \PSH(X,\theta)\\ [\varphi] \in \mathcal A, \ \varphi|_Y\not\equiv -\infty} } \int_Y (\theta|_Y+\ddc \varphi|_Y)^m.
\end{equation}
Here $[u]\in \mathcal{A}$ simply means that $u$ has analytic singularity type (recall \eqref{eq: analyt_sing_type_def}).

In case $Y \subset X$ is not contained in the non-nef locus of $\{\theta\}$,  Collins--Tosatti \cite{CT22} extend the above definition of restricted volume, making connection with the notion of movable intersection due to Boucksom \cite{Bo02}:
\begin{equation}\label{eq: rest_vol_CT_def}
        \vol_{X|Y}(\{\theta\})\coloneqq \lim_{\varepsilon \searrow 0}\sup_{\substack{\varphi\in \PSH(X,\theta + \varepsilon \omega)\\ [\varphi] \in \mathcal A, \ \varphi|_Y\not\equiv -\infty}} \int_Y (\theta|_Y + \varepsilon \omega|_Y +\ddc \varphi|_Y)^m.
\end{equation}

In what follows we give interpretations of $\vol_{X|Y}(\{\theta\})$ using the trace operator, while revisiting \cite{CT22, His12, Mat13}. 

\begin{proposition}\label{prop:restvol_voltrace}
Assume that $Y$ is not contained in the non-K\"ahler locus of $\{\theta\}$, then
    \begin{equation}\label{eq:volrest}
        \vol_{X|Y}(\{\theta\})= \int_Y \left( \theta|_Y+\ddc \Tr_Y^{\theta}(V_{\theta}) \right)^m =\int_Y (\theta|_Y + \ddc V_\theta|_Y)^m.
    \end{equation}
\end{proposition}

The equality $ \vol_{X|Y}(\{\theta\}) =\int_Y (\theta|_Y + \ddc V_\theta|_Y)^m$ was obtained in \cite[Remark 2.8]{CT22}, and implicitly in \cite{His12, Mat13}. 

\begin{proof}
We start with the first equality of \eqref{eq:volrest}. Since $Y$ is not contained in the non-K\"ahler locus of $\{\theta\}$, $V_\theta |_ Y \not \equiv -\infty$, hence also $\nu(V_{\theta},Y)=0$.

    Take a quasi-equisingular approximation $\varphi_j\in \PSH(X,\theta+\epsilon_j\omega)$ of $V_{\theta}$. Since $\varphi_j\geq V_{\theta}\succeq \varphi$ for any $\varphi\in \PSH(X,\theta)$ with analytic singularity type such that $\varphi|_Y\not\equiv -\infty$, we have $\varphi_j|_Y \succeq \varphi|_Y$, hence by the monotonicity theorem of \cite{WN19},
    \[
    \int_Y (\theta|_Y+\epsilon_j\omega+\ddc \varphi_j|_Y)^m\geq \int_Y (\theta|_Y+\ddc\varphi|_Y)^m.
    \]
    Letting $j\to \infty$ and applying \cref{lma:Traceasquasiequiapp}, we conclude that the $\geq$ direction in the first equality of \eqref{eq:volrest}.
    
   For the reverse direction, by definition, for any fixed $\epsilon>0$, we have
    \[
    \int_Y (\theta|_Y+\epsilon_j\omega|_Y+\ddc\varphi_j|_Y)^m\leq \int_Y (\theta|_Y+\epsilon\omega|_Y+\ddc\varphi_j|_Y)^m\leq \vol_{X|Y}(\{\theta\}+\epsilon\{\omega\})
    \]
    for all large enough $j$. Letting $j\to\infty$ and $\epsilon \searrow 0$, using the continuity of $\vol_{X|Y}$ (\cite[Corollary~4.11]{Mat13}) together with \cref{lma:Traceasquasiequiapp}, we conclude the first equality of \eqref{eq:volrest}.

Now we address the second equality. Due to \cite[Theorem~1.1]{WN19}, the defining formula \eqref{eq: rest_vol_Mat_def}, and the definition of $V_\theta$, we obtain that
\[
\vol_{X|Y}(\{\theta\}) \leq \int_Y (\theta|_Y + \ddc V_\theta|_Y)^m.
\]
The reverse equality now follows from the first equality of \eqref{eq:volrest}, \cite[Theorem~1.1]{WN19} and the fact that $V_\theta|_Y \preceq \Tr_Y^\theta(V_\theta)$ (Remark~\ref{rem:naiverestprecRest}).
\end{proof}

As we show now, the first formula of the previous proposition holds away from the non-nef locus of $\{\theta\}$ as well:
\begin{theorem}\label{thm: rest_volume_tr_volume} If $Y$ is not contained in the non-nef locus of $\{\theta\}$, i.e. $\nu(V_\theta, Y)=0$, then 
\begin{flalign}\label{eq: non-nef_rest_vol}
\vol_{X|Y}(\{\theta\}) &= \lim_{\varepsilon \searrow 0}\int_Y (\theta + \varepsilon \omega +  \ddc V_{\theta+ \varepsilon \omega})|_Y^m=\lim_{\varepsilon \searrow 0}\int_Y (\theta|_Y + \varepsilon \omega|_Y +  \ddc \Rest_Y^{\theta+ \varepsilon \omega}(V_{\theta+ \varepsilon \omega}) )^m \nonumber \\
&=\int_Y \left( \theta|_Y+\ddc \Tr_Y^{\theta}(V_{\theta}) \right)^m. 
\end{flalign}
\end{theorem}

The main point is the last equality, the first equality is already obtained in \cite[Remark 2.2]{CT22}. 

\begin{proof} Since $\nu(V_\theta, Y)=0$, $Y$ is not contained in the non-K\"ahler locus of $\{\theta + \varepsilon \omega\}$ for all $\varepsilon>0$. As a result, due to \eqref{eq: rest_vol_CT_def} and \eqref{eq:volrest} only the last equality of \eqref{eq: non-nef_rest_vol} needs to be argued.

For any $\varepsilon_0 >0$, since $V_\theta \leq V_{\theta + \varepsilon_0 \omega}$,  due to Lemma \ref{lem: trace_vol_prop}(i) we have
\begin{equation}\label{eq: easy_est}
\Rest^\theta_Y(V_\theta) = \lim _{\varepsilon \searrow 0} \Rest^{\theta + \varepsilon \omega}_Y(V_\theta) \leq \Rest^{\theta + \varepsilon_0 \omega}_Y(V_\theta) \leq \Rest^{\theta + \varepsilon_0 \omega}_Y(V_{\theta + \varepsilon_0 \omega}).
\end{equation}
By Proposition~\ref{prop: vol_limit_model} we have that $\int_X (\theta + \varepsilon \omega)^n_{V_{\theta + \varepsilon \omega}} \searrow \int_X \theta_{V_\theta}^n$ as $\epsilon\searrow 0$. Since 
\[
\int_X (\theta + \varepsilon \omega)^n_{V_{\theta + \varepsilon \omega}} \geq \int_X (\theta + \varepsilon \omega)_{V_\theta}^n\geq  \int_X \theta_{V_\theta}^n>0, 
\]
for $\varepsilon>0$ small enough \cite[Lemma~4.3]{DDNLmetric}  implies the existence of $w_\varepsilon \in \PSH(X,\theta + \varepsilon \omega)$ and $\alpha_\varepsilon \searrow 0$ such that
\[
(1-\alpha_\varepsilon)V_{\theta + \varepsilon \omega} + \alpha_\varepsilon w_\varepsilon \leq V_\theta.
\]

Fixing $\varepsilon_0>0$, we apply $P_{\theta + \varepsilon_0 \omega}[\cdot]$ for the potentials on both sides of the above inequality for $\varepsilon < \varepsilon_0$. Using concavity of this operator, we obtain that 
\[
(1-\alpha_\varepsilon)P_{\theta + \varepsilon_0 \omega}[V_{\theta + \varepsilon \omega}] + \alpha_\varepsilon P_{\theta + \varepsilon_0 \omega}[w_\varepsilon] \leq P_{\theta + \varepsilon_0 \omega}[V_\theta] \leq P_{\theta + \varepsilon_0 \omega}[V_{\theta + \varepsilon \omega}].
\]
Letting $\varepsilon \searrow 0$ we obtain that 
\begin{equation}\label{eq: P_I_limit}
P_{\theta + \varepsilon_0 \omega}[V_{\theta + \varepsilon \omega}]_\mathcal I= P_{\theta + \varepsilon_0 \omega}[V_{\theta + \varepsilon \omega}] \searrow P_{\theta + \varepsilon_0 \omega}[V_{\theta}] = P_{\theta + \varepsilon_0 \omega}[V_{\theta}]_\mathcal I,
\end{equation}
where the first and last equalities follow from \cite[Theorem~3.8]{DX21}. Indeed, there exists an increasing sequence $\chi_j^\varepsilon \in \PSH_>(X,\theta + \varepsilon_0 \omega)$ with analytic singularity type such that $\chi_j^\varepsilon \nearrow V_{\theta + \varepsilon \omega}$ (see the beginning of \cite[Section~4]{Gup24} for this construction). Due to \cite[Lemma~4.1]{DDNLmetric}, we have that $d^{\theta + \varepsilon_0 \omega}_\mathcal S([\chi^\varepsilon_j],[V_{\theta + \varepsilon \omega}]) \to 0$ (we refer to \cite{DDNLmetric} for the definition and properties of $d_{\mathcal{S}}$), hence 
 \cite[Theorem~3.8]{DX21} is indeed applicable.

Using \eqref{eq: P_I_limit} we can apply Lemma~\ref{lem: trace_decreasing_convergence}, to conclude that $\Rest_Y^{\theta + \varepsilon_0 \omega}(V_{\theta + \varepsilon \omega})  \searrow \Rest_Y^{\theta + \varepsilon_0 \omega}(V_{\theta})$ as $\varepsilon \searrow 0$. Using Proposition~\ref{prop: vol_limit_model} and Lemma~\ref{lem: trace_vol_prop}(ii) 
 we get that 
\[
\int_Y (\theta|_Y + \varepsilon_0 \omega|_Y)^m_{\Rest_Y^{\theta + \varepsilon_0 \omega}(V_{\theta + \varepsilon \omega}) } \searrow \int_Y (\theta|_Y + \varepsilon_0 \omega|_Y)^m_{\Rest_Y^{\theta + \varepsilon_0 \omega}(V_{\theta})} \geq \int_Y (\theta|_Y)^m_{\Rest_Y^{\theta}(V_{\theta})}  \ \ \textup{ as } \varepsilon \searrow 0.
\]
Due to Lemma~\ref{lem: trace_vol_prop}(ii) again, $\int_Y (\theta|_Y + \varepsilon_0 \omega|_Y)^m_{\Rest_Y^{\theta + \varepsilon_0 \omega}(V_{\theta})} \searrow \int_Y (\theta|_Y)^m_{\Rest_Y^{\theta}(V_{\theta})}$ as $\varepsilon_0 \searrow 0$. 
This implies that for some sequence $\tilde \epsilon_j$ satisfying $\frac{1}{j} \geq \tilde \epsilon_j \searrow 0$ we have that 
\begin{equation}\label{eq: limit_1}
\int_Y \left(\theta|_Y + \frac{1}{j} \omega|_Y\right)^m_{\Rest_Y^{\theta + \frac{1}{j} \omega}(V_{\theta + \tilde \varepsilon_j \omega}) } \searrow  \int_Y (\theta|_Y)^m_{\Rest_Y^{\theta}(V_{\theta})}  \quad \textup{ as } j \to \infty.
\end{equation}
Moreover, due to Lemma~\ref{lem: trace_vol_prop}(ii) (applied to $u = V_{\theta + \tilde \varepsilon_j \omega}$) and  \eqref{eq: easy_est} we have
\begin{equation}\label{eq: limit_2}
\int_Y \left(\theta|_Y + \frac{1}{j} \omega|_Y\right)^m_{\Rest_Y^{\theta|_Y + \frac{1}{j} \omega|_Y}(V_{\theta + \tilde \varepsilon_j \omega}) } \geq \int_Y \left(\theta|_Y + \tilde \varepsilon_j \omega|_Y\right)^m_{\Rest_Y^{\theta + \tilde \varepsilon_j\omega}(V_{\theta + \tilde \varepsilon_j \omega}) } \geq \int_Y (\theta|_Y)^m_{\Rest_Y^{\theta}(V_{\theta})}.
\end{equation}
Putting \eqref{eq: limit_1} and \eqref{eq: limit_2} together, it results that 
\[
c_j \coloneqq \int_Y (\theta|_Y + \tilde \varepsilon_j \omega|_Y)^m_{\Rest_Y^{\theta + \tilde \varepsilon_j \omega}(V_{\theta + \tilde \varepsilon_j \omega}) } \searrow \int_Y (\theta|_Y)^m_{\Rest_Y^{\theta}(V_{\theta})}. 
\]
As $\lim_{j \to \infty} c_j = \vol_{Y|X}(\{\theta\})$, due to the first two identities of \eqref{eq: non-nef_rest_vol}, the proof is finished.
\end{proof}

\paragraph{An $L^2$ extension theorem for the trace operator.}
The trace operator allows to obtain the following sheaf theoretic $L^2$ extension theorem:
\begin{theorem}\label{thm:OT}
    Let $u\in \PSH(X,\theta)$ with $\nu(u,Y)=0$. Then for any $\lambda>0$, we have
    \begin{equation}\label{eq:OT}
        \mathcal{I}(\lambda \Rest_Y^\theta(\varphi))\subseteq \mathcal{I}(\lambda \varphi)|_Y.
    \end{equation}
\end{theorem}
The restriction $\mathcal{I}(\lambda \varphi)|_Y$ is the inverse image ideal sheaf  given by $\mathcal{I}(\lambda \varphi)/(\mathcal{I}(\lambda \varphi)\cap \mathcal{I}_Y)$, where $\mathcal{I}_Y$ is the ideal sheaf defining $Y$. In other words,  $\mathcal{I}(\lambda \varphi)|_Y $ is the sheaf of holomorphic functions on $Y$ with local extensions in  $\mathcal{I}(\lambda \varphi)$, and it is different from the pull-back of $\mathcal{I}(\lambda \varphi)$ to $Y$ (see \cite[Page~163]{Har} for a detailed discussion).

\begin{proof} 
Let $u_j\in \PSH(X,\theta + \varepsilon_j \omega)$ be the quasi-equisingular approximation of $u$ from \cref{thm:Demailly}. From \cref{def: trace_general} we get that $\Rest_Y^\theta(u) \preceq u_j|_Y$.
For any $\lambda'>\lambda>0$, we can find $j>0$ so that
    \[
    \mathcal{I}(\lambda' u_j)\subseteq \mathcal{I}(\lambda u).
    \]
    By the Ohsawa--Takegoshi extension theorem applied to $\varphi_j$ \cite[Theorem~14.1]{Dem12}, we have
    \[
    \mathcal{I}(\lambda' \Rest_Y^\theta(u))\subseteq \mathcal{I}(\lambda' u_j|_Y)\subseteq \mathcal{I}(\lambda'u_j)|_Y\subseteq \mathcal{I}(\lambda u)|_Y.
    \]
    Letting $\lambda'\to \lambda$ and applying the strong openness theorem of Guan--Zhou \cite{GZ15} to $\Rest_Y^\theta(u)$, we conclude \eqref{eq:OT}.
\end{proof}

\section{Applications to restricted volumes}

In this section we assume that $X$ is a connected projective manifold of dimension $n$. 
Let $A$ be a very ample line bundle on $X$. Take a smooth positive Hermitian metric $g$ on $A$ and let $\omega=-\ddc \log g=\frac{\mathrm{i}}{2\pi}\Theta_g>0$ be our background K\"ahler form.

We also fix a big line bundle $L$ with a smooth Hermitian metric $h$. Let $\theta=-\ddc \log h=\frac{\mathrm{i}}{2\pi}\Theta_h$ be the curvature form. We will consider a singular psh metric $h\mathrm{e}^{-u}$ on $L$ with $u\in \PSH(X,\theta)$. 
We also fix an arbitrary twisting line bundle $T$ on $X$.

Let $Y \subset X$ be a connected $m$-dimensional complex submanifold. 
When $Y\neq X$, we can always find $\psi_Y\in \QPSH(X)$ such that $\{\psi_Y = -\infty\} = Y$, $\psi_Y$ has neat analytic singularity type and is log canonical. As elaborated in \cite[Lemma~2.3]{Fin22}, we can choose $\psi_Y$ so that in a neighborhood of $Y$ we have
\begin{equation}\label{eq: Psi_Y_def}
\psi_Y(x) = 2(n-m) \log \textup{dist}(x, Y),
\end{equation}
for some Riemannian distance function $\textup{dist}(\cdot, Y)$.
For more details we refer to \cite{Fin22}. This choice of $\psi_Y$ allows interpreting $\mathcal I_Y$, the sheaf of holomorphic functions vanishing along $Y$, as a multiplier ideal sheaf:

\begin{lemma}\label{lma:IpsiY}
Assume that $Y\neq X$.
    The multiplier ideal sheaf of $\psi_Y$ can be calculated as
    \begin{equation}\label{eq:mis_psi}
    \mathcal{I}(\psi_Y)=\mathcal{I}_Y.
    \end{equation}
    Moreover, given $y\in Y$ and $\epsilon>0$, for any germ $f\in \mathcal{I}_{Y,y}$ we have
    \begin{equation}\label{eq:integrabilitypsiY}
    \int_U |f|^{\epsilon}\mathrm{e}^{-\psi_Y}\omega^n<\infty,
    \end{equation}
    where $U$ is an open neighbourhood of $y$ in $X$.
\end{lemma}
\begin{proof}
    Since $\psi_Y$ is locally bounded away from $Y$, it suffices to prove \eqref{eq:mis_psi} along $Y$. Fix $y\in Y$, and we will verify \eqref{eq:mis_psi} germ-wise at $y$.

Take an open neighbourhood $U \subset X$ of $y$ and a biholomorphic map $F\colon U\rightarrow V\times W$, where $V$ is an open neighbourhood of $y$ in $Y$ and $W$ is a connected open subset in $\mathbb{C}^{n-m}$ containing $0$, such that $F(Y\cap U)=V\times\{0\}$. For any $x\in U$, write $x_V,x_W$ for the two components of $F(x)$ in $V$ and $W$ respectively. We denote the coordinates in $\mathbb{C}^{n-m}$ as $w_1,\ldots,w_{n-m}$.
    
   Due to \eqref{eq: Psi_Y_def}, after possibly shrinking $U$, we may assume that 
    \[
    \exp(-\psi_Y(x))=|x_W|^{2m-2n}+\mathcal{O}(1)
    \]
    for any $x\in U\setminus Y$.
    
    Given $f\in \mathcal{I}_{Y,y}$, after shrinking $U$, we may assume that there exists $g_1,\ldots,g_{n-m}\in H^0(V\times W,\mathcal{O}_{V\times W})$ such that
    \[
    f=\sum_{i=1}^{n-m} w_i g_i.
    \]
    In order to verify $f\in \mathcal{I}(\psi_Y)_y$, it suffices to show $w_ig_i \in \mathcal{I}\left( (\sum_{i=1}^{n-m} |w_i|^2)^{m-n}\right)_{F(y)}$, which follows from Fubini's theorem. The proof of \eqref{eq:integrabilitypsiY} is similar.

    Conversely, take $f\in \mathcal{I}(\psi_Y)$, the similar application of Fubini's theorem shows that after possible shrinking $U$, we have $f|_{Y}=0$. By Rückert's Nullstellensatz \cite[Page~67]{CAS}, it follows that $f\in \mathcal{I}_Y$.
\end{proof}

The following Ohsawa--Takegoshi extension result is a special case of \cite[Theorem~1.4]{His12}. Using an approximation argument, it can also be derived from \cite[Theorem 2.8]{Dem16ext}.

\begin{theorem}\label{thm: OT_ext} Suppose that $u \in \PSH(X,\theta - \delta \omega)$ for some $\delta>0$ and $u|_Y\not\equiv -\infty$. Fix a Hermitian metric $r$ on $T$.
Then there exists $k_0(\delta,r)>0$ such that for all $k \geq k_0$ and  $s \in H^0(Y, \mathcal{O}_Y(L)^k \otimes \mathcal{O}_Y(T) \otimes \mathcal I(ku|_Y))$, there exists an extension $\tilde s \in H^0(X, \mathcal{O}_X(L)^k \otimes \mathcal{O}_X(T)  \otimes \mathcal I(ku))$ such that
\[
\int_X (h^k\otimes r)(\tilde s,\tilde s) \mathrm{e}^{-ku}\,\omega^n \leq  C \int_Y (h^k\otimes r)(s,s) \mathrm{e}^{-ku|_Y}\,\omega|_Y^m,
\]
where $C > 0$ is an absolute constant, independent of the data $(u,s,k)$.
\end{theorem}

In the next lemma, we prove \eqref{eq:restr_volume_intr_2} for potentials with analytic singularity type:

\begin{lemma}\label{lem: analytic_formula} Assume that $u$ has analytic singularity type and $\theta_u$ is a K\"ahler current. Suppose that $u|_Y \not\equiv -\infty$. Then
\begin{equation}\label{eq:asymanasing}
\int_Y (\theta|_Y+\ddc u|_Y)^m = \lim_{k \to \infty} \frac{m!}{k^m} \dim_{\mathbb C}\left\{ s|_Y:  s\in H^0(X,\mathcal{O}_X(L)^k \otimes \mathcal{O}_X(T) \otimes \mathcal I(k u))\right\}.
\end{equation}
\end{lemma}

\begin{proof} 
 We may assume that $Y\neq X$ as otherwise the result follows from \cite[Theorem~1.1]{DX21}.
Suppose that $\varepsilon\in (0,1)$ is small enough so that $(1-\varepsilon)u \in \PSH(X,\theta)$.

In the estimates below we use the sheaf $\mathcal{J}(\cdot)$ recalled and studied in \cref{lem: coh_J} above. Using \cite[Theorem~1.1]{DX21} we can start to write the following sequence of inequalities:
\[
\begin{aligned}
&\frac{1}{m!}\int_Y (\theta|_Y+\ddc u|_Y)^m \\
=& \lim_{k \to \infty} \frac{1}{k^m} h^0(Y,\mathcal{O}_Y(L)^k\otimes \mathcal{O}_Y(T) \otimes \mathcal I(k u|_Y)) \\
\leq& \varliminf_{k \to \infty} \frac{1}{k^m} \dim\left\{ s|_Y :  s\in H^0(X,\mathcal{O}_X(L)^k \otimes \mathcal{O}_X(T) \otimes \mathcal I(k u))\right\}\quad \text{by \cref{thm: OT_ext}}\\
\leq& \varlimsup_{k \to \infty} \frac{1}{k^m} \dim\left\{ s|_Y :  s\in H^0(X,\mathcal{O}_X(L)^k \otimes \mathcal{O}_X(T) \otimes \mathcal I(k u))\right\} \\
\leq& \varlimsup_{k \to \infty} \frac{1}{k^m} \dim \left\{s|_Y :  s \in H^0(X,\mathcal{O}_X(L)^k \otimes \mathcal{O}_X(T) \otimes \mathcal J((1-\varepsilon)ku))\right\} \  \ \ \  \text{(see below)}\\
\leq& \varlimsup_{k \to \infty} \frac{1}{k^m} \dim_{\mathbb C}\left\{ s \in H^0(Y,\mathcal{O}_Y(L)^k\otimes \mathcal{O}_Y(T) ) : \log h^k (s,s) \leq (1-\varepsilon) k u|_Y\right\} \\
\leq& \varlimsup_{k \to \infty} \frac{1}{k^m} h^0\left(Y,\mathcal{O}_Y(L)^k \otimes \mathcal{O}_Y(T) \otimes \mathcal I((1-\varepsilon) k u|_Y)\right) \\
=&\frac{1}{m!}\int_Y \left(\theta|_Y+(1-\varepsilon)\ddc u|_Y\right)^m \quad \text{by \cite[Theorem~1.1]{DX21}},
\end{aligned}
\]
where in the fifth line we used  \cite[Proposition~4.1.6]{Dem15}.
Letting $\varepsilon \to 0$, \eqref{eq:asymanasing} follows from multilinearity of the non-pluripolar product.
\end{proof}

Next we prove \eqref{eq:restr_volume_intr_1} for K\"ahler currents:

\begin{proposition}\label{prop: rest_volume} Let $u \in \PSH_{>}(X,\theta)$ with $\nu(u,Y) =0$. Then
\begin{equation}\label{eq:relativeDX}
    \int_Y \left(\theta|_Y+\ddc \Rest_Y^\theta(u)\right)^m = \lim_{k \to \infty} \frac{m!}{k^m} h^0(Y,\mathcal{O}_Y(L)^k \otimes \mathcal{O}_Y(T) \otimes \mathcal I(ku)|_Y).
\end{equation}
\end{proposition}

\begin{proof} 
We may assume that $Y\neq X$, as otherwise \eqref{eq:relativeDX} reduces immediately to \cite[Theorem~1.1]{DX21}.

Let $u_j \searrow u$ be a quasi-equisingular approximation of $u$, given by Theorem \ref{thm:Demailly}. Thus, we can assume that the  $\theta_{u_j}$ are K\"ahler currents. 
After possibly replacing $u_j$ by a subsequence, there exists $\varepsilon_0 \in (0,1)\cap \mathbb{Q}$ such that $\theta_{(1-\varepsilon)^2 u_j}$ and $\theta_{(1-\varepsilon) u_j}$ are also K\"ahler currents for any $\varepsilon \in (0,\varepsilon_0)$.

We claim that for any $k\in \mathbb{N}$,
\begin{equation}\label{eq:JcapI}
\mathcal{J}((1-\epsilon)ku_j)\cap \mathcal{I}(\psi_Y)\subseteq \mathcal{I}((1-\epsilon)^2ku_j+\psi_Y).
\end{equation}
Recall the sheaf $\mathcal{J}$ is introduced and studied in \cref{lem: coh_J}.

Take $x\in X$, and it suffices to argue \eqref{eq:JcapI} along the germ of $x$. Since $\psi_Y$ is locally bounded outside $Y$, we may assume that $x\in Y$. Recall that by \cref{lma:IpsiY}, $\mathcal{I}(\psi_Y)=\mathcal{I}_Y$.

Let $f \in \mathcal J((1-\varepsilon)k u_j)_x\cap \mathcal{I}(\psi_Y)_x$. Then there is an open neighbourhood $U$ of $x$ in $X$ such that
$|f|^{2 (1-\varepsilon)} \mathrm{e}^{-k(1-\varepsilon)^2 u_j} \leq C$
holds on $U\setminus \{u_j=-\infty\}$ for some $C>0$, hence
\[
\int_U |f|^2 \mathrm{e}^{-k (1-\varepsilon)^2 u_j-\psi_Y}\,\omega^n = \int_U |f|^{2 (1-\varepsilon)} \mathrm{e}^{-k (1-\varepsilon)^2 u_j} |f|^{2 \varepsilon} \mathrm{e}^{-\psi}\,\omega^n \leq C \int_U |f|^{2 \varepsilon} \mathrm{e}^{-\psi}\,\omega^n < \infty,
\]
where the last inequality follows from \cref{lma:IpsiY}. We have proved the claim \eqref{eq:JcapI}.

Next we consider the following composition morphism of coherent sheaves on $Y$:
\begin{flalign}\label{eq: sheaf_injection}
\mathcal J((1-\varepsilon)k u_j)|_Y \coloneqq \frac{\mathcal J((1-\varepsilon)k u_j)}{\mathcal J((1-\varepsilon)k u_j) \cap \mathcal{I}_Y} \hookrightarrow \frac{\mathcal I((1-\varepsilon )^2 ku_j)}{\mathcal J((1-\varepsilon)k u_j) \cap \mathcal I_Y} \to \frac{\mathcal I((1-\varepsilon )^2 ku_j)}{\mathcal I((1-\varepsilon)^2 k u_j + \psi_Y)}.
\end{flalign}
Above we identified the coherent $\mathcal{O}_X$-modules supported on $Y$ with coherent $\mathcal{O}_Y$-modules. Note that the target of \eqref{eq: sheaf_injection} is also supported on $Y$ as $\psi_Y$ is locally bounded outside $Y$. We denote the coherent $\mathcal{O}_Y$-module whose pushforward to $X$ gives $\frac{\mathcal I((1-\varepsilon )^2 ku_j)}{\mathcal I((1-\varepsilon)^2 k u_j + \psi_Y)}$ by $\mathcal{I}_{k,j}$.

In \eqref{eq: sheaf_injection}, the first map is the inclusion and the second one is the obvious projection induced by \eqref{eq:JcapI}. Although in general the second map fails to be injective, we observe that the composition is still injective as $\mathcal I((1-\varepsilon)^2 k u_j + \psi)\subseteq \mathcal I(\psi)=\mathcal{I}_Y$. Therefore, for any $k\in \mathbb{N}$, we have an injective morphism of coherent $\mathcal{O}_Y$-modules:
\begin{equation}\label{eq:injLkTideal}
\mathcal{O}_Y(L)^k\otimes \mathcal{O}_Y(T)\otimes \mathcal J((1-\varepsilon)k u_j)|_Y \hookrightarrow \mathcal{O}_Y(L)^k\otimes \mathcal{O}_Y(T)\otimes \mathcal{I}_{k,j}.
\end{equation}
By the projection formula,
\begin{equation}\label{eq:projfor}
H^0\left(Y,\mathcal{O}_Y(L)^k\otimes \mathcal{O}_Y(T)\otimes \mathcal{I}_{k,j}\right)\cong H^0\left(X,\mathcal{O}_X(L)^k\otimes \mathcal{O}_X(T)\otimes \frac{\mathcal I((1-\varepsilon )^2 ku_j)}{\mathcal I((1-\varepsilon)^2 k u_j + \psi_Y)}\right).
\end{equation}

Using \cite[Theorem~1.1]{DX21} we can start the following inequalities:
\[
\begin{aligned}
&\frac{1}{m!}\int_Y \left(\theta|_Y+\ddc\Rest_Y^\theta(u)\right)^m \\
=& \lim_{k \to \infty} \frac{1}{k^m} h^0(Y, \mathcal{O}_Y(L)^k\otimes \mathcal{O}_Y(T) \otimes \mathcal I(k\Rest_Y^\theta(u))) \quad \text{by \cite[Theorem~1.1]{DX21}}\\
\leq & \varliminf_{k \to \infty} \frac{1}{k^m} h^0(Y, \mathcal{O}_Y(L)^k\otimes \mathcal{O}_Y(T) \otimes \mathcal I(ku)|_Y) \quad \text{ by Theorem \ref{thm:OT}}\\
\leq & \varlimsup_{k \to \infty} \frac{1}{k^m} h^0(Y, \mathcal{O}_Y(L)^k\otimes \mathcal{O}_Y(T) \otimes \mathcal I(ku)|_Y)\\
\leq & \varlimsup_{k \to \infty} \frac{1}{k^m} h^0(Y, \mathcal{O}_Y(L)^k\otimes \mathcal{O}_Y(T) \otimes \mathcal I(ku_j)|_Y) \\
\leq & \varlimsup_{k \to \infty} \frac{1}{k^m} h^0(Y, \mathcal{O}_Y(L)^k\otimes \mathcal{O}_Y(T) \otimes \mathcal J((1-\varepsilon )ku_j)|_Y)  \quad \text{by \cite[Proposition~4.1.6]{Dem15}}\\
\leq & \varlimsup_{k \to \infty} \frac{1}{k^m} h^0(Y, \mathcal{O}_Y(L)^k\otimes \mathcal{O}_Y(T) \otimes \mathcal{I}_{k,j})\quad \text{by \eqref{eq:injLkTideal}}\\
\leq & \varlimsup_{k \to \infty} \frac{1}{k^m} \dim_\mathbb C \Big\{ s|_Y : \ s \in H^0\Big(X, \mathcal{O}_X(L)^k \otimes \mathcal{O}_X(T) \otimes \frac{ \mathcal I ((1-\varepsilon )^2 ku_j)}{\mathcal I((1-\varepsilon)^2 k u_j + \psi_Y)}\Big)\Big\} \quad \text{by \eqref{eq:projfor}}\\
= & \varlimsup_{k \to \infty} \frac{1}{k^m} \dim_\mathbb C \left\{s|_Y : s\in H^0(X,\mathcal{O}_X(L)^k \otimes \mathcal{O}_X(T)\otimes \mathcal{I}((1-\epsilon)^2ku_j))\right\} \quad \text{(see below)}\\
=& \frac{1}{m!}\int_Y \left(\theta|_Y+(1-\varepsilon)^2 \ddc u_j|_Y\right)^m \quad \text{by \cref{lem: analytic_formula}},
\end{aligned}
\]
where in the penultimate line we used \cite[Theorem~1.1(6)]{CDM17} for $q=0$. Letting $\varepsilon \to \infty$ and then $j \to \infty$ the result follows.
\end{proof}

Next, we obtain \eqref{eq:restr_volume_intr_1} in full generality:

\begin{theorem}\label{thm: relativeDX}
    Let $u\in \PSH(X,\theta)$ such that $\nu(u,Y) =0$. Then
\begin{equation}\label{eq:relativeDX2}
    \int_Y \left(\theta|_Y+\ddc \Rest_Y^\theta(u)\right)^m = \lim_{k \to \infty} \frac{m!}{k^m} h^0(Y,\mathcal{O}_Y(L)^k \otimes \mathcal{O}_Y(T) \otimes \mathcal I(ku)|_Y).
\end{equation}
\end{theorem}
When $Y=X$, in view of \cref{ex:traceYequalX}, formula \eqref{eq:relativeDX2} recovers the main result of \cite{DX21}.

\begin{proof}
Using \cref{thm:OT} and \cite[Theorem~1.1]{DX21} we obtain that 

\begin{flalign*}\label{eq:DX_cor}
    \int_Y \left(\theta|_Y+\ddc \Rest_Y^\theta(u)\right)^m &= \lim_{k \to \infty} \frac{m!}{k^m} h^0(Y,\mathcal{O}_Y(L)^k \otimes \mathcal{O}_Y(T) \otimes \mathcal I(k\Rest_Y^\theta(u)))\\
    & \leq \varliminf_{k \to \infty} \frac{m!}{k^m} h^0(Y,\mathcal{O}_Y(L)^k \otimes \mathcal{O}_Y(T) \otimes \mathcal I(ku)|_Y).
\end{flalign*}

Now we address the other direction in \eqref{eq:relativeDX2}. Let $\phi \in H^0(X,A)$ be a section that does not vanish identically on $Y$. Such $\phi$ exists since $A$ is very ample. Recall that we chose a hermitian metric $g$ on $A$ with curvature equal to $\omega$.
   
We fix $k_0 \in \mathbb{N}$.  For any $k \geq 0$, we have that $k = q k_0 + r$ with $q,r \in \mathbb{N}$ and $r \in \{0,\ldots, k_0-1\}$. Also, we have an injective linear map
\[
H^0(Y,\mathcal{O}_Y(L)^k \otimes \mathcal{O}_Y(T) \otimes \mathcal I(k u|_Y))\xrightarrow{\cdot \phi^{\otimes q}} H^0\left(Y,\mathcal{O}_Y(L)^k \otimes \mathcal{O}_Y(A)^q\otimes \mathcal{O}_Y(T) \otimes \mathcal I(k u|_Y)\right).
\]
So
\begin{equation}
     h^0\left(Y,\mathcal{O}_Y(L)^k \otimes \mathcal{O}_Y(T) \otimes \mathcal I(k u|_Y)\right)\leq h^0\left(Y,\mathcal{O}_Y(L)^k \otimes \mathcal{O}_Y(A)^q\otimes \mathcal{O}_Y(T) \otimes \mathcal I(k u|_Y)\right).
\end{equation}
Using this estimate can start the following sequences of inequalities
\begin{flalign*}
 &\varlimsup_{k\to\infty} \frac{m!}{k^m}  h^0\left(Y,\mathcal{O}_Y(L)^k \otimes \mathcal{O}_Y(T) \otimes \mathcal I(k u|_Y)\right)  \\
 \leq &  \varlimsup_{k\to\infty} \frac{m!}{k^m}    h^0\left(Y,\mathcal{O}_Y(L)^k \otimes \mathcal{O}_Y(A)^q\otimes \mathcal{O}_Y(T) \otimes \mathcal I(k u|_Y)\right) \\
 = &  \frac{1}{k_0^m}  \varlimsup_{q\to\infty} \frac{m!}{q^m}    h^0\left(Y,\mathcal{O}_Y(L)^{q k_0} \otimes \mathcal{O}_Y(A)^q\otimes \mathcal{O}_Y(T) \otimes \mathcal{O}_Y(L)^{r}\otimes \mathcal I(k u|_Y)\right)\\
 \leq & \frac{1}{k_0^m}  \varlimsup_{q\to\infty} \frac{m!}{q^m}    h^0\left(Y,\mathcal{O}_Y(L)^{q k_0} \otimes \mathcal{O}_Y(A)^q\otimes \mathcal{O}_Y(T) \otimes \mathcal{O}_Y(L)^{r}\otimes \mathcal I(k_0q u|_Y)\right)\\
 = & \int_Y \left(\theta|_Y+\frac{1}{k_0}\omega|_Y+\ddc \Rest_Y^{\theta+k_0^{-1}\omega}(u)\right)^m,
\end{flalign*}
where in the fourth line we have used that $k_0 q \leq k$ and in the last line we have used \cref{prop: rest_volume} for the big line bundle $L^{k_0} \otimes A$, the K\"ahler current $k_0 \theta_u - \ddc \log g = k_0 \theta_u + \omega$, and twisting bundle $T \otimes L^r$. Letting $k_0\to\infty$ and using \cref{lem: trace_vol_prop}(ii), we conclude that
\[
\varlimsup_{k\to\infty} \frac{m!}{k^m}  h^0\left(Y,\mathcal{O}_Y(L)^k \otimes \mathcal{O}_Y(T) \otimes \mathcal I(k u|_Y)\right)\leq \int_Y \left(\theta|_Y+\ddc \Rest_Y^\theta(u)\right)^m.
\]
\end{proof}

Lastly, we turn our attention to global sections. For this we will need the following global $L^2$ extension theorem for the trace operator:

\begin{theorem}\label{thm: OT_ext_global} Suppose that $\theta_u$ is a K\"ahler current and $\nu(u,Y)=0$. Then there exists $k_0$ such that for all $k \geq k_0$ and  $s \in H^0(Y, \mathcal{O}_Y(L)^k\otimes \mathcal{O}_Y(T) \otimes \mathcal I(k \Rest_Y^\theta(u)))$, there exists an extension $\tilde s \in H^0(X, \mathcal{O}_X(L)^k \otimes \mathcal{O}_X(T)  \otimes \mathcal I(ku))$.
\end{theorem}

\begin{proof} We may assume that $Y\neq X$ and that $\theta_u \geq 3\delta \omega$ for some $\delta>0$. Let $u^D_j$ be the decreasing quasi-equisingular approximation of $u$ from \cref{thm:Demailly}. We can assume that $\theta_{u_j^D} \geq 2\delta \omega$ for all $j$ as follows from \cref{rmk:Demapp}. Also, there exists $\varepsilon_0 >0$ such that $\theta_{(1+\varepsilon)u_j^D} \geq \delta \omega$ for any $\varepsilon \in (0,\varepsilon_0)$. Take $k_0=k_0(\delta)$ as in \cref{thm: OT_ext}.

We fix $k \geq k_0$ and $s \in H^0(Y, \mathcal{O}_Y(L)^k\otimes \mathcal{O}_Y(T) \otimes \mathcal I(k\Rest_Y^\theta(u)))$. By the Guan--Zhou's strong openness theorem \cite{GZ15}, there exists $\varepsilon  \in (0,\varepsilon_0)$ such that $s \in H^0(Y, \mathcal{O}_Y(L)^k\otimes \mathcal{O}_Y(T) \otimes \mathcal I(k(1+\varepsilon)\Rest_Y^\theta(u)))$. 

Since $\Rest_Y^\theta(u) \preceq u_j^D|_Y$, we obtain that  $s \in H^0(Y, \mathcal{O}_Y(L)^k\otimes \mathcal{O}_Y(T) \otimes \mathcal I(k(1+\varepsilon)u_j^D|_Y))$. Due to \cref{thm: OT_ext} there exists $\tilde s_j \in H^0(X, \mathcal{O}_X(L)^k \otimes \mathcal{O}_X(T)  \otimes \mathcal I(k(1+\varepsilon)u_j^D))$ such that $\tilde s_j|_Y =s$, for all $j$.

But by definition of quasi-equisingular approximation in \cref{def:equising}, we obtain that for high enough $j$ the inclusion $\mathcal I(k(1+\varepsilon)u_j^D) \subset \mathcal I(k u)$ holds. As a result, $\tilde s_j \in H^0(X, \mathcal{O}_X(L)^k \otimes \mathcal{O}_X(T)  \otimes \mathcal I(k u))$ for high enough $j$, finishing the argument.
\end{proof}

As an application of the above results, we obtain  \eqref{eq:restr_volume_intr_2}:

\begin{theorem}\label{thm: rest_volume_2} Let $u \in \PSH_{>}(X,\theta)$ such that $\nu(u,Y) =0$. Then
\begin{equation*}
    \int_Y \left(\theta|_Y+\ddc \Rest_Y^\theta(u)\right)^m  = \lim_{k \to \infty} \frac{m!}{k^m} \dim_{\mathbb C}\left\{ s|_Y:  s\in H^0(X,\mathcal{O}_X(L)^k \otimes \mathcal{O}_X(T) \otimes \mathcal I(k u))\right\}.
\end{equation*}
\end{theorem}

\begin{proof} 
This is a consequence of \cite[Theorem~1.1]{DX21}, \cref{thm: OT_ext_global} and \cref{thm: relativeDX}:
\begin{flalign*}
    \int_Y \left(\theta|_Y+\ddc \Rest_Y^\theta(u)\right)^m  & = \lim_{k \to \infty} \frac{m!}{k^m} h^0(Y,\mathcal{O}_Y(L)^k\otimes \mathcal{O}_Y(T) \otimes \mathcal I(k \Rest_Y^\theta(u)))\\ 
    & \leq  \varliminf_{k \to \infty} \frac{m!}{k^m} \dim_{\mathbb C}\left\{ s|_Y:  s\in H^0(X,\mathcal{O}_X(L)^k \otimes \mathcal{O}_X(T) \otimes \mathcal I(k u))\right\}\\
    & \leq  \varlimsup_{k \to \infty} \frac{m!}{k^m} \dim_{\mathbb C}\left\{ s|_Y:  s\in H^0(X,\mathcal{O}_X(L)^k \otimes \mathcal{O}_X(T) \otimes \mathcal I(k u))\right\}\\
  & \leq  \lim_{k \to \infty} \frac{m!}{k^m} h^0(Y ,\mathcal{O}_Y(L)^k\otimes \mathcal{O}_Y(T) \otimes \mathcal I(k u)|_Y)\\
  & =  \int_Y \left(\theta|_Y+\ddc \Rest_Y^\theta(u)\right)^m.
\end{flalign*}
\end{proof}

Lastly, we give a new proof of the equality between transcendental and algebraic restricted volumes from \cite[Theorem 1.3]{Mat13} (c.f. \cite[Theorem 1.3]{His12}):

\begin{corollary}\label{cor: rest_volume_2} Suppose that $Y$ is not contained in the augmented base locus of $L$. Then
\begin{equation}\label{eq:relativeDX5}
\vol_{X|Y}(\{\theta\})  = \lim_{k \to \infty} \frac{m!}{k^m} \dim_{\mathbb C}\left\{ s|_Y:  s\in H^0(X,\mathcal{O}_X(L)^k )\right\}.
\end{equation}
\end{corollary}
\begin{proof}There exists $u_j \in \PSH_>(X,\theta)$ with analytic singularity type such that $u_j \nearrow V_\theta$ almost everywhere. Existence of such a sequence is a consequence of \cite[Theorem~3.2]{DP04}, as elaborated in the discussion at the beginning of \cite[Section~4]{Gup24}.

Since $Y$ is not contained in the non-K\"ahler locus of $\{\theta\}$, we can assume that $u_j|_Y \not \equiv -\infty$. From Theorem~\ref{thm: rest_volume_2} and \cref{ex:analytic} we obtain that
\begin{flalign*}
    \int_Y \left(\theta|_Y+\ddc u_j|_Y\right)^m & =     \int_Y \left(\theta|_Y+\ddc \Rest_Y^\theta(u_j)\right)^m \\
    & =  \lim_{k \to \infty} \frac{m!}{k^m} \dim_{\mathbb C}\left\{ s|_Y:  s\in H^0(X,\mathcal{O}_X(L)^k \otimes \mathcal I(k u_j ))\right\}\\
& \leq \varliminf_{k \to \infty} \frac{m!}{k^m} \dim_{\mathbb C}\left\{ s|_Y:  s\in H^0(X,\mathcal{O}_X(L)^k )\right\}
\end{flalign*}
By \cite[Theorem~2.3]{DDNL18mono} and Proposition~\ref{prop:restvol_voltrace}  we conclude that 

\begin{flalign*}
 \vol_{X|Y}(\{\theta\}) =   \int_Y \left(\theta|_Y+\ddc V_\theta|_Y\right)^m &   \leq \varliminf_{k \to \infty} \frac{m!}{k^m} \dim_{\mathbb C}\left\{ s|_Y:  s\in H^0(X,\mathcal{O}_X(L)^k )\right\}
\end{flalign*}

To argue the reverse inequality in  \eqref{eq:relativeDX5}, we use Proposition~\ref{prop:restvol_voltrace} and apply \cite[Theorem~1.1]{DX21} to $\Rest_Y^\theta(V_\theta)$:
\begin{flalign*}
\vol_{X|Y}(\{\theta\}) =     \int_Y \left(\theta|_Y+\ddc \Rest_Y^\theta(V_\theta)\right)^m   = \lim_{k \to \infty} \frac{m!}{k^m} h^0(Y,\mathcal{O}_Y(L)^k\otimes \mathcal I(k \Rest_Y^\theta(V_\theta)))
\end{flalign*}
We are finished if we can argue that $s|_Y \in H^0(Y,\mathcal{O}_Y(L)^k\otimes \mathcal I(k \Rest_Y^\theta(V_\theta)))$ for any $ s\in H^0(X,\mathcal{O}_X(L)^k )$. This simply follows from the fact that $h^k (s,s) \mathrm{e}^{-k V_\theta}$  is bounded above on $X$ for each $k$, hence so is $h^k (s,s) \mathrm{e}^{-k \Rest_Y^\theta(V_{\theta})}$ on $Y$ thanks to \cref{rem:naiverestprecRest}.
\end{proof}

\printbibliography

\small
\noindent {\sc Department of Mathematics, University of Maryland, College Park, USA}\\
{\tt tdarvas@umd.edu}\vspace{0.1in}

\noindent {\sc  Institute of Geometry and Physics, USTC, Hefei, China}\\
{\tt xiamingchen2008@gmail.com }\vspace{0.1in}

\end{document}